\documentclass[dvips,ps]{imsart}
\usepackage[OT1]{fontenc}
\usepackage{amsthm,amsmath,natbib,amssymb}
\usepackage[colorlinks]{hyperref}
\usepackage{hypernat}
\usepackage{graphicx}
\pubyear{2006}
\volume{0}
\issue{0}
\firstpage{1}
\lastpage{8}

\startlocaldefs
\numberwithin{equation}{section}
\theoremstyle{plain}
\newtheorem{thm}{Theorem}
\newtheorem{cor}{Corollary}
\endlocaldefs

\def \beqn {\begin{eqnarray}}
\def \beqnn {\begin{eqnarray*}}
\def \eeqn {\end{eqnarray}}
\def \eeqnn {\end{eqnarray*}}
\def \GG {\mathcal{G}}
\def \FF {\mathcal{F}}

\def \FF {\mathcal{F}}
\def \no {\Arrowvert}
\def \ind {\hbox{ 1\hskip -3pt I}}
\def \R {\mathbb{R}}

\def \P  {\mathbb{P}} 
\def \R {\mathbb{R}}
\def \E {\mathbb{E}}
\def \N {\mathbb{N}}
\newtheorem{lemma}{Lemma}
\newtheorem{defi}{Definition}

\begin{document}

\begin{frontmatter}
\title{Statistical learning with indirect observations}
\runtitle{Statistical learning with indirect observations}

\begin{aug}
\author{S\'ebastien Loustau,}
\runauthor{S. Loustau}

\affiliation{Universit\'e d'Angers, LAREMA}

\address{Universit\'e d'Angers, LAREMA\\ loustau@math.univ-angers.fr}

\end{aug}
\begin{abstract}
Let $(X,Y)\in\mathcal{X}\times \mathcal{Y}$ be a random couple with unknown distribution $P$. Let $\GG$ be a class of measurable functions  and $\ell$ a loss function. The problem of statistical learning deals with the estimation of the Bayes:
$$g^*=\arg\min_{g\in\GG}\E_P \ell(g(X),Y).
$$
In this paper, we study this problem when we deal with a contaminated sample $(Z_1,Y_1),\ldots , (Z_n,Y_n)$ of i.i.d. indirect observations. Each input $Z_i$, $i=1,\ldots ,n$ is distributed from a density $Af$, where $A$ is a known compact linear operator and $f$ is the density of the direct input $X$. \\
We derive fast rates of convergence for empirical risk minimizers based on regularization methods, such as deconvolution kernel density estimators or spectral cut-off. These results are comparable to the existing fast rates in \cite{kolt} for the direct case. It gives some insights into the effect of indirect measurements in the presence of fast rates of convergence.

\end{abstract}
\end{frontmatter}

\section{Introduction}
In many real-life situations, direct data are not available and measurement errors occur. In many examples, such as medecine, astronomy, econometrics or meteorology, these measurement errors should not be neglected. Let us consider the following example from signal processing in oncology. Medical images (such as scanner, magnitude resonance imaging) play an increasingly important role in diagnosing and treating cancer patients. In the clinical setting, imaging data allows to better evaluate whether a cancer patient is responding to therapy and to adjust the therapy accordingly. In such a setting, the response variable could be the total response to the treatment, a partial response or the absence of a response. However, image interpretation and management in clinical trials triggers a number of issues such as doubtful reliability of image analysis due to a high variability in image interpretation, censoring bias, and a number of operational issues due to complex image data workflow. Consequently, biomarkers, such as bidimensional measurements of lesions, suffer from measurement errors. For these reasons, statistical learning with indirect observations may play a crucial role for this problem.

In this contribution, we address this problem in the general statistical learning context. The model can be described through 4 components:

\begin{itemize}
\item a generator \textbf{G} of random variables $X\in\mathcal{X}\subseteq \R^d$ with unknown density $f$ with respect to $\nu$, a $\sigma$-finite measure defined on $\mathcal{X}$,
\item a supervisor \textbf{S} who associates to $X$ an output $Y\in\mathcal{Y}$, according to an unknown conditional probability,
\item a known linear compact operator \textbf{A}$:L_2(\nu,\mathcal{X})\to L_2(\nu,\tilde{\mathcal{X}})$ which corrupts $X$ given $Z$ where $Z$ has density $Af$ with respect to $\nu$,
\item a Learning Machine \textbf{LM} which given $n$ i.i.d. observations $(Z_i,Y_i)$ returns an estimator $\hat{y}$ associated to any given $x$ from the generator.
\end{itemize}

\begin{minipage}{0.6\linewidth}
\begin{center}
\includegraphics[width=8cm]{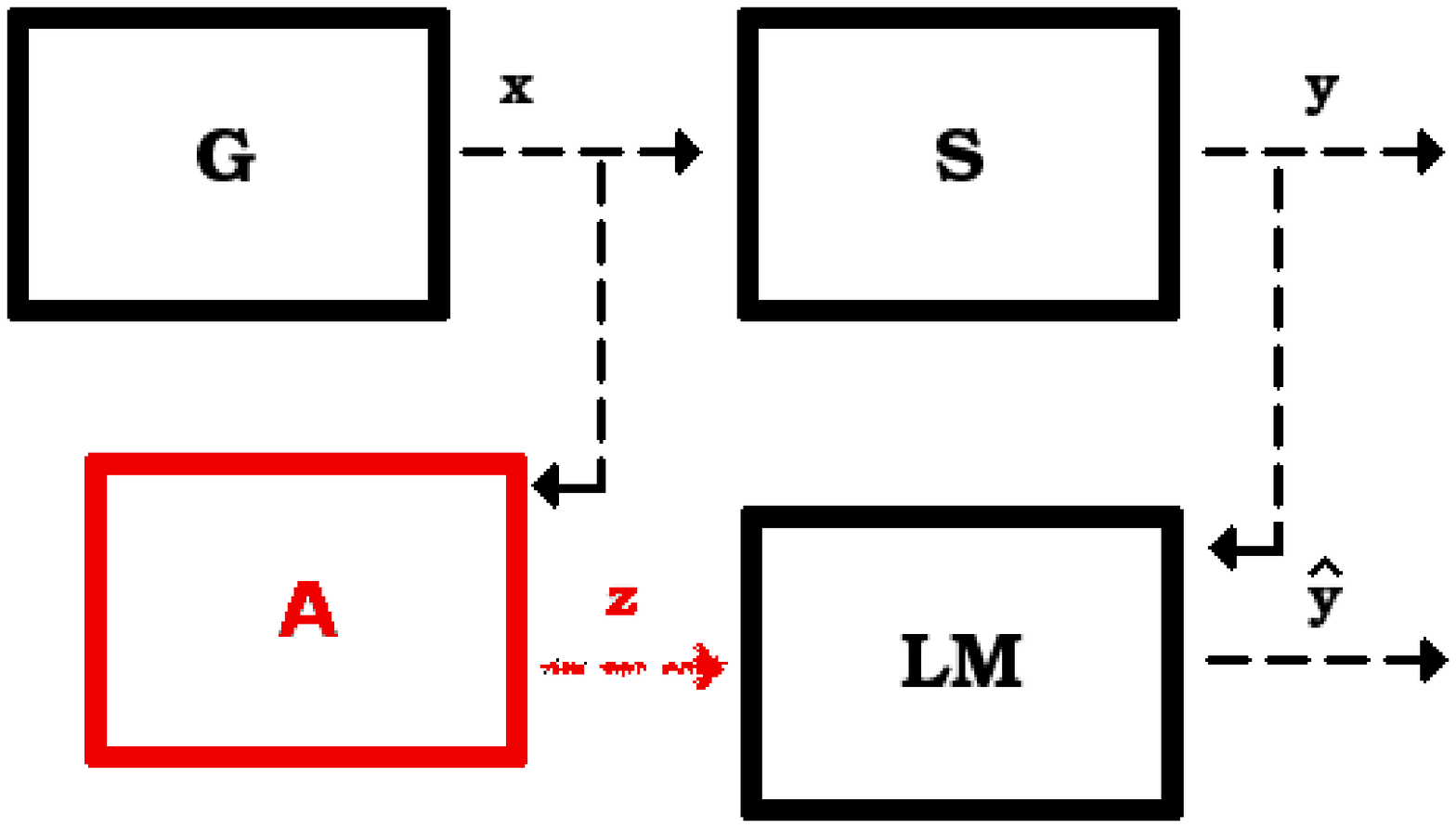}
\end{center}
\end{minipage}\hfill
\begin{minipage}{0.4\linewidth}
\textit{Figure 1. This representation has its origin in \cite{vapnik2000}. Here, the presence of the nuisance operator $A$ makes the matter an inverse problem}. 
\end{minipage}

The aim is to design a decision rule which returns, for each new generator's value $x$, a value $\hat{y}$ as close as possible to the supervisor's response $y$. Note that depending on the nature of the supervisor, Figure 1 contains models of classification, density estimation or regression. 

The more extensively studied model with indirect observations is the additive measurement error. In this case, we observe indirect inputs:
$$
Z_i=X_i+\epsilon_i,i=1,\ldots , n,
$$
where $(\epsilon_i)_{i=1}^n$ are i.i.d. with known density $\eta$. It corresponds in Figure 1 to a convolution operator $A_\eta:f\mapsto f*\eta$ and we are faced to classification with errors in variables, density deconvolution, or regression with errors in variables.\\


For these purposes, we introduce a bounded loss function $\ell:\R\times \mathcal{Y}\to [0,1]$ and a class $\GG$ of measurable functions $g:\mathcal{X}\to \R$. To define the best approximation, the problem is to choose from the given set of functions $g\in\GG$, the one that minimizes the risk functional:
\beqn
\label{risk}
R_{\ell}(g)=\E_P \ell(g(X),Y).
\eeqn
The performances of a given $g$ are measured through its non-negative excess risk, given by:
\beqn
\label{excess}
R_{\ell}(g)-R_{\ell}(g^*),
\eeqn
where $g^*$ is the minimizer over $\GG$ of the risk \eqref{risk}. It is important to point out that we do not adress in this paper the problem of model selection of $\GG$. It consists in studying the difference $R_{\ell}(g^*)-\inf_{g }R_{\ell}(g)$, where the infimum is taken over all possible measurable functions $g$. Here, the target $g^*$ corresponds to the oracle in the family $\GG$. The purpose of this work is to use Empirical Risk Minimization (ERM) strategies based on a corrupted sample to minimize the excess risk \eqref{excess}. \\

In the direct case, as we observe i.i.d. $(X_1,Y_1),\ldots,(X_n,Y_n)$ with law $P$, a classical way is to consider the ERM estimator defined as:
\beqn
\label{erm}
\hat{g}_n=\arg\min_{g\in\GG}R_n(g),
\eeqn
where $R_n(g)$ denotes the empirical risk:
\beqnn
R_n(g)=\frac{1}{n}\sum_{i=1}^n\ell(g(X_i),Y_i)=P_n\ell(g).
\eeqnn
In the sequel, the empirical measure of the direct sample $(X_1,Y_1),\ldots,(X_n,Y_n)$ will be denoted as $P_n$.  A large literature (see \cite{vapnik2000} for such a generality) deals with the statistical performances of \eqref{erm} in terms of the excess risk \eqref{excess}. To be concise, under complexity assumptions over $\GG$ (such as finite VC dimension (\cite{vapnik82}), entropy conditions (\cite{vdg}), Rademacher complexity assumptions (\cite{kolt}), it is possible to get both consistency and rates of convergence of ERM estimators (see also \cite{nedelec} in classification). The main probabilistic tool is the statement of uniform concentration of the empirical measure to the true measure. It comes from the so-called Vapnik's bound:
\beqn
\label{vb}
R_{\ell}(\hat{g}_n)-R_{\ell}(g^*)&\leq& R_{\ell}(\hat{g}_n)-R_n(\hat{g}_n)+R_n(g^*)-R_{\ell}(g^*)\notag\\
&\leq & 2\sup_{g\in\GG}|(P_n-P)l(g)|.
\eeqn
It is important to highlight that \eqref{vb} can be improved using a local approach (see \cite{toulouse}). It consists in reducing the supremum to a neighborhood of $g^*$. We do not develop these important refinements in this introduction for the sake of concision whereas it is the main ingredient of the literature cited above. It allows to get fast rates of convergence in pattern recognition.\\
 
Here, the framework is essentially different. Given a linear compact operator $A$, we observe a corrupted sample $(Z_1,Y_1),\ldots,(Z_n,Y_n)$ where $Z_i,$ $i=1,\ldots ,n$ are i.i.d. with density $Af$. As a result, the empirical measure $P_n=\frac{1}{n}\sum_{i=1}^n\delta_{(X_i,Y_i)}$ is unobservable and standard ERM \eqref{erm} is not available. Unfortunately, using the contaminated sample $(Z_1,Y_1),\ldots,(Z_n,Y_n)$ in standard ERM \eqref{erm} fails:
\beqnn
\frac{1}{n}\sum_{i=1}^nl(g(Z_i),Y_i)\longrightarrow \E l(g(Z),Y)\not= R_{\ell}(g).
\eeqnn
Due to the action of $A$, the empirical measure from the indirect sample, denoted by $\tilde{P}_n=\frac{1}{n}\sum_{i=1}^n\delta_{(Z_i,Y_i)}$, differs from $P_n$ (in the sequel, we also note as $\tilde{P}$ the corresponding true measure of $(Z,Y)$). We are facing an ill-posed inverse problem. This problem has been recently considered in \cite{pinkfloyds} for discriminant analysis with errors in variables.

In this work, we suggest a comparable strategy in statistical learning. Given a smoothing parameter $\lambda=(\lambda_1,\ldots ,\lambda_d)\in\R^d_+$, we consider the following $\lambda$-Empirical Risk Minimization ($\lambda$-ERM):
\beqn
\label{lerm}
\arg\min_{g\in\GG}R_n^\lambda(g),
\eeqn
where $R_n^\lambda(g)$ is defined in a general way as:
\beqn
\label{ler}
R_n^\lambda(g)=\int_\mathcal{X} l(g(x),y)\hat{P}_\lambda(dx,dy).
\eeqn
The measure $\hat{P}_\lambda=\hat{P}_\lambda(Z_1,Y_1,\ldots,Z_n,Y_n)$ is data-dependent to the set of indirect inputs $(Z_1,\ldots,Z_n)$. It will be related to standard regularization methods coming from the inverse problem literature (see \cite{engle}).As a consequence, it depends on a smoothing parameter $\lambda\in\R^d_+$. An explicit construction of $\hat{P}_\lambda$ and the empirical risk \eqref{ler} is detailled in Section 2 in pattern recognition with applications in Section 3.\\

To study the performances of the minimizer $\hat{g}_n^\lambda$ of the empirical risk \eqref{ler}, it is possible to use empirical processes theory in the spirit of \cite{vdg,wvdv} or more recently \cite{kolt}. Following \eqref{vb}, in the presence of indirect observations, we can write\footnote{where with a slight abuse of notations, we write:
$$
(R_{\ell}-R^\lambda_\ell)(g-g')=R_{\ell}(g)-R_{\ell}(g')-R_{\ell}^\lambda(g)+R_{\ell}^\lambda(g').
$$}:
\beqn
\label{nvb}
R_{\ell}(\hat{g}_n^\lambda)-R_{\ell}(g^*)&\leq& R_{\ell}(\hat{g}_n^\lambda)-R_n^\lambda(\hat{g}_n^\lambda)+R_n^\lambda(g^*)-R_{\ell}(g^*)\notag\\
&\leq& R_{\ell}^\lambda(\hat{g}_n^\lambda)-R_n^\lambda(\hat{g}_n^\lambda)+ R_n^\lambda(g^*)-R_{\ell}^\lambda(g^*)+ (R_{\ell}-R_{\ell}^\lambda)(\hat{g}_n^\lambda -g^*)\notag\\
&\leq & \sup_{g\in\GG}|(R_n^\lambda-R_{\ell}^\lambda)(g^*-g)|+\sup_{g\in\GG}|(R_{\ell}^\lambda-R_{\ell})(g-g^*)|,
\eeqn
where in the sequel, under integrability conditions and using Fubini: 
\beqn
\label{rl}
R_{\ell}^\lambda(g)=\E R_n^\lambda(g)= \int \ell(g(x),y)\E \hat{P}_\lambda(dx,dy).
\eeqn
Bound \eqref{nvb} is called Inverse Vapnik's bound. It consists in two terms: 
\begin{itemize}
\item A variance term $\sup_{g\in\GG}|(R_n^\lambda-R_{\ell}^\lambda)(g^*-g)|$ related to the estimation of $g^*$: this term can be controlled thanks to uniform exponential inequalities such as Talagrand's concentration inequality, applied to a class of functions depending on a parameter.
\item A bias term $\sup_{g\in\GG}|(R_{\ell}^\lambda-R_{\ell})(g-g^*)|$: it comes from the estimation of $P$ into the expression of $R_{\ell}(g)$ with estimator $\hat{P}_\lambda$. This term is specific to our method. However, it seems to be related to the usual bias term in nonparametric density estimation. Indeed, we can see easily that:
\beqnn
R_{\ell}^\lambda(g)-R_{\ell}(g)=\int \ell(g(x),y)[\E \hat{P}_\lambda-P_\lambda](dx,dy).
\eeqnn
\end{itemize}
The choice of $\lambda$ is crucial in the decomposition \eqref{nvb}. We will show below that the variance term exploses when $\lambda$ tends to zero whereas the bias term vanishes. Parameter $\lambda$ has to be chosen as a trade-off between these two terms, and as a consequence will depend on unknown parameters. The problem of adaptation is not adressed in this paper but it is an interesting future direction.\\

In this work, we consider $\mathcal{Y}=\{0,1,\ldots,M\}$ for $M\geq 1$. In other words, we study the model of classification with indirect observations (see \cite{jaune} for a survey in the direct case). The contribution is organized as follows. In Section 2, we propose to give an explicit construction of the empirical risk \eqref{ler} in classification thanks to the set of indirect observations. We state a general upper bound for the solution of the $\lambda$-ERM \eqref{lerm} under minimal assumptions over the loss function $\ell$ and the complexity of $\mathcal{G}$. It gives a generalization of the results of \cite{kolt} when dealing with indirect observations. Section 3 gives applications of the result of Section 2 in two particular settings. In the errors-in-variables case, we generalize the results of \cite{pinkfloyds}. For the general case, we use projection in the spectrum of operator $A$. We state rates of convergence which generalize the existing fast rates of convergence pointed out by \cite{kolt}. There coincide with a recent lower bound proposed in discriminant analysis by \cite{pinkfloyds}. Section 4 is devoted to a discussion related to the complexity assumption when we deal with indirect observations whereas Section 5 concludes the paper. Section 6 is dedicated to the proofs of the main results.

\section{General Upper Bound}
\label{sec2}
In this section, we detail the construction of the empirical risk \eqref{ler} in classification. We give minimal assumptions to control the expected excess risk \eqref{excess} of the procedure. The construction of the empirical risk is based on the following decomposition of the true risk:
\beqn
\label{classifrisk}
R_{\ell}(g)=\sum_{y\in\mathcal{Y}} p(y)\int_\mathcal{X} \ell(g(x),y)f_y(x)\nu(dx),
\eeqn
where $f_y(\cdot)$ is the conditional density of $X|Y=y$ and $p(y)=\P(Y=y)$, for any $y\in\mathcal{Y}=\{0,\ldots,M\}$. With such a decomposition, we suggest to estimate each $f_y(\cdot)$ using a nonparametric density estimator. To state a general upper bound, we consider a family of estimators such as:
\beqn
\label{est}
\forall y\in\mathcal{Y},\hat{f}_y(x)=\frac{1}{n_y}\sum_{i=1}^{n_y}k_{\lambda}(Z^y_i,x),
\eeqn
where $n_y=\mathrm{card}\{i:Y_i=y\}$, $k_\lambda:\mathcal{\tilde{X}}\times\mathcal{X}\to \R$ and the set of inputs $(Z_i^y)_{i=1}^{n_y}=\{Z_i,i=1,\ldots,n:Y_i=y\}$.\\
Here, we consider a constant bandwidth $\lambda$ for any $y\in\mathcal{Y}$ in $\hat{f}_y$. It illustrates rather well the difference of our approach with plug-in type estimators (see \cite{AT} for instance). If we want to estimate $f_y$, for each $y\in\mathcal{Y}$, the bandwidth $\lambda$ in \eqref{est} has to depend on $n_y$ and the regularity of $f_y$. However, the aim is to estimate the true risk $R_{\ell}(g)$. To get satisfying upper bounds, we will see that $\lambda$ does not necessary depend on the value $y\in\mathcal{Y}$.

It is also important to remark that assumption \eqref{est} provides a variety of nonparametric estimators of $f_y$. For instance, if $Af=f*\eta$ is a convolution operator, we can construct a deconvolution kernel provided that the noise has a nonnull Fourier transform. This is a rather classical approach in deconvolution problems (see \cite{Fan} or \cite{meister}). Another standard example of \eqref{est} is to consider projection estimators of the conditional densities using the SVD of operator $A$ or many other regularization methods (see \cite{engle}). Section \ref{application} describes these examples.

Finally we plug estimators \eqref{est} in the true risk \eqref{classifrisk} to get an empirical risk defined as:
\beqnn
R_n^\lambda(g)=\sum_{y\in\mathcal{Y}} \int_\mathcal{X} \ell(g(x),y)\hat{f}_y(x)\nu(dx)\hat{p}(y),
\eeqnn
where $\hat{p}(y)=\frac{n_y}{n}$ is an estimator of the quantity $p(y)=\P(Y=y)$. Thanks to \eqref{est}, this empirical risk can be written as:
\beqn
\label{ler2}
R_n^\lambda(g)=\frac{1}{n}\sum_{i=1}^n\ell_\lambda(g,(Z_i,Y_i)),
\eeqn
where $\ell_\lambda(g,(z,y))$ is a modified version of $\ell(g(x),y)$ given by:
\beqnn
\ell_\lambda(g,(z,y))=\int_\mathcal{X} \ell(g(x),y)k_{\lambda}(z,x)\nu(dx).
\eeqnn
\\

In this section, we study general upper bounds for the expected excess risk of the estimator:
\beqn
\label{decerm}
\hat{g}_n^\lambda=\arg\min \frac{1}{n}\sum_{i=1}^nl_\lambda(g,(Z_i,Y_i)).
\eeqn
In case no such minimum exists, we can consider a $\delta$-approximate minimizer as in \cite{empimini} without significant change in the results. \\
The main idea is to use iteratively a deviation inequality for suprema of empirical processes due to \cite{bousquet}. It allows to control the increments of the empirical process:
\beqnn
\nu_n^\lambda(g)=\frac{1}{\sqrt{n}}\sum_{i=1}^n \left(\ell_\lambda(g,(Z_i,Y_i))-\E \ell_\lambda(g,(Z,Y))\right).
\eeqnn 
Here, it is important to note that Talagrand's type inequality has to be applied to the class of functions $\{(z,y)\mapsto \ell_\lambda(g,(z,y)),g\in\GG\}$. This class depends on a regularization parameter $\lambda$. This parameter will be calibrated as a function of $n$ and that's why the deviation inequality has to be used carefully. For this purpose, we introduce in Definition 1 particular classes $\{l_\lambda(g),g\in\GG\}$.
\begin{defi}
\label{LB}
We say that the class $\{\ell_\lambda(g),g\in\GG\}$ is a LB-class (Lipschitz bounded class) with respect to $\mu$ with parameters $(c(\lambda),K(\lambda))$ if these two properties hold:
\begin{description}
\item[(L$_\mu$)] $\{\ell_\lambda(g),g\in\GG\}$ is Lipschitz w.r.t. $\mu$ with constant $c(\lambda)$:
\beqnn
\forall g,g'\in\GG,\,\Arrowvert \ell_\lambda(g)-\ell_\lambda(g')\Arrowvert_{L_2(\tilde{P})}\leq c(\lambda)\Arrowvert \ell(g)-\ell(g')\Arrowvert_{L_2(\mu)}.
\eeqnn
\item[(B)] $\{\ell_\lambda(g),g\in\GG\}$ is uniformly bounded with constant $K(\lambda)$:
$$
\sup_{g\in\GG}\sup_{(z,y)}| \ell_\lambda(g,(z,y))|\leq K(\lambda).
$$
\end{description}
\end{defi}
A LB-class of loss function is Lipschitz and bounded with constants which depend on $\lambda$. These properties are necessary to derive explicitly the upper bound of the variance in \eqref{nvb} as a function of $\lambda$.

More precisely, the Lipschitz property \textbf{(L$_\mu$)} is a key ingredient to control the complexity of the class of functions $\{\ell_\lambda(g),g\in\GG\}$. In the sequel, we use the following geometric complexity parameter:
$$
\tilde{\omega}_n(\GG,\delta,\mu)=\E\sup_{g,g'\in\mathcal{G}:
\no \ell(g)-\ell(g')\no_{L_2(\mu)}\leq\delta}\left|(\tilde{P}-\tilde{P}_n)(\ell_\lambda(g)-\ell_\lambda(g'))\right|.
$$ 
The control of such a quantity is proposed in Section \ref{imc} thanks to standard entropy conditions related to the class $\mathcal{G}$.

Finally \textbf{(B)} is necessary to apply Bousquet's inequality to the class of functions $\{\ell_\lambda(g)-\ell_\lambda(g'),\,g\in\mathcal{G}\}$, which depends on the smoothing parameter $\lambda$. This condition could be relaxed by dint of recent advances on empirical processes in an unbounded framework (see \cite{nonexactlecue} or \cite{lederer}).

\begin{defi}
\label{KKb}
For $\kappa\geq 1$, we say that $\mathcal{F}$ is a Bernstein class with respect to $\mu$  with parameter $\kappa$ if there exists $\kappa_0\geq 0$ such that  for every $f\in\FF$:
\beqnn
\no f\no_{L_2(\mu)}^2\leq \kappa_0[\E_P f]^{\frac{1}{\kappa}}.
\eeqnn
\end{defi}
This assumption first appears in \cite{empimini} for $\mu=P$ when $\mathcal{F}=\{\ell(g)-\ell(g'),g,g'\in\mathcal{G}\}$ is the excess loss class. It allows to control the excess risk in statistical learning using functional's Bernstein inequality such as Talagrand's type inequality. It goes back to the standard margin assumption in classification (see \cite{mammen,tsybakov2004}), where in this case $\kappa=\frac{\alpha+1}{\alpha}$ for a so-called margin parameter $\alpha\geq 0$.

Definition \ref{KKb} has to be combined with the Lipschitz property of Definition \ref{LB}. It allows us to have the following serie of inequalities:
\beqn
\label{serie}
\Arrowvert \ell_\lambda(g)-\ell_\lambda(g^*)\Arrowvert_{L_2(\tilde{P})}\leq c(\lambda)\Arrowvert f\Arrowvert_{L_2(\mu)}\leq c(\lambda)\left(\E_Pf\right)^{\frac{1}{2\kappa}},
\eeqn
where $f\in\FF=\{\ell(g)-\ell(g^*),\,g\in\mathcal{G}\}$ is the excess loss class. 

Last definition provides a control of the bias term in \eqref{nvb} as follows:
\begin{defi}
\label{app}
The class $\{\ell_\lambda(g),g\in\GG\}$  has approximation function $a(\lambda)$ and residual constant $0<r<1$ if the following holds:
\beqnn
\forall g\in\GG,\,(R_{\ell}-R^\lambda_l)(g-g^*)\leq a(\lambda)+r(R_{\ell}(g)-R_{\ell}(g^*)),
\eeqnn
where with a slight abuse of notations, we write:
$$
(R_{\ell}-R^\lambda_\ell)(g-g^*)=R_{\ell}(g)-R_{\ell}(g^*)-R_{\ell}^\lambda(g)+R_{\ell}^\lambda(g^*).
$$
\end{defi}
This definition warrants a control of the bias in the Inverse Vapnik's bound \eqref{nvb}. It is straightforward that with Definition \ref{app}, we get a control of the excess risk as follows:
\beqnn
R_{\ell}(\hat{g}_n^\lambda)-R_{\ell}(g^*)&\leq& \frac{1}{1-r}\left(\sup_{g\in\GG(1)}|(\tilde{P}_n-\tilde{P})(\ell_\lambda(g)-\ell_\lambda(g^*))|+a(\lambda)\right),
\eeqnn
where in the sequel:
$$
\mathcal{G}(\delta)=\{g\in\mathcal{G}:R_{\ell}(g)-R_{\ell}(g^*)\leq \delta\}.
$$
Explicit functions $a(\lambda)$ and residual constant $r<1$ are obtained in Section 3. There depend on the regularity conditions and allow to get rates of convergence.

We are now on time to state the main result of this section.
\begin{thm}
\label{mainresult}
Consider a LB-class $\{\ell_\lambda(g),g\in\GG\}$ with respect to $\mu$ with parameters $(c(\lambda),K(\lambda))$ and approximation function $a(\lambda)$ such that:
\beqn
\label{biascontrol}
a(\lambda)\leq C_1 \left(\frac{c(\lambda)}{\sqrt{n}}\right)^{\frac{2\kappa}{2\kappa+\rho -1}}\mbox{ and }
K(\lambda)\leq \frac{c(\lambda)^{\frac{2\kappa}{2\kappa+\rho-1}}n^{\frac{\kappa+\rho-1}{2\kappa+\rho-1}}}{1+\log n},
\eeqn
for some $C_1>0$.\\
Suppose $\{\ell(g)-\ell(g^*),g\in\mathcal{G}\}$ is Bernstein with respect to $\mu$ with parameter $\kappa>1$ where $g^*\in\arg\min_\mathcal{G} R_{\ell}(g)$ is unique. Suppose there exists $0<\rho<1$ such that for every $\delta>0$:
\beqn
\label{modulus}
\tilde{\omega}_n(\GG,\delta,\mu)=\E\sup_{g,g'\in\mathcal{G}:\no \ell(g)-\ell(g')\no_{L_2(\mu)}\leq\delta}|\tilde{P}-\tilde{P}_n|(\ell_\lambda(g)-\ell_\lambda(g'))\leq C_2 \frac{c(\lambda)}{\sqrt{n}}\delta^{1-\rho},
\eeqn
for some $C_2>0$.\\
Then estimator $\hat{g}_n^\lambda$ defined in \eqref{decerm} satisfies, for $n$ great enough:
\beqnn 
\E R_{\ell}(\hat{g}_n^\lambda)-R_{\ell}(g^*)\leq C\left(\frac{c(\lambda)}{\sqrt{n}}\right)^{\frac{2\kappa}{2\kappa+\rho -1}},
\eeqnn
where $C=C(C_1,C_2,\kappa,\kappa_0,\rho)>0$.
\end{thm}
\hspace{-0.6cm} The proof of this result is presented in Section 6. Here follows some remarks.

This upper bound generalizes the result presented in \cite{kolt} to the indirect framework. Theorem \ref{mainresult} provides rates of convergence  $\left(c(\lambda)/\sqrt{n}\right)^{2\kappa/2\kappa+\rho-1}$. In the noise-free case, with standard ERM estimators,  \cite{tsybakov2004,kolt} obtain fast rates $n^{-\kappa/2\kappa+\rho-1}$. In the presence of contaminated inputs, rates are slower since $c(\lambda)\to +\infty$ as $n\to +\infty$. Hence, the price to pay for the inverse problem is quantified by the Lipschitz constant $c(\lambda)$ in Definition \ref{LB}.

The behavior of constants $c(\lambda)$ depend on the difficulty of the inverse problem through the degree of ill-posedness of operator $A$. Section \ref{application} proposes to deal with midly ill-posed inverse problems. In this case, $c(\lambda)$ depend polynomially on $\lambda$. 

The Lipschitz property introduced in Definition \ref{LB} is central. Gathering with the complexity assumption \eqref{modulus}, it leads to a control of the variance term in decomposition \eqref{nvb}. The first statement of condition \eqref{biascontrol} gives the order of the bias term. It leads to the excess risk bound.

The second part of \eqref{biascontrol} is due to the use of a deviation's inequality from \cite{bousquet} to the class $\{l_\lambda(g),g\in\mathcal{G}\}$. In Section 3, we give explicit constants $c(\lambda)$ and $K(\lambda)$. It appears that this assumption is always guaranteed.

The control of the modulus of continuity in \eqref{modulus} is specific to the indirect framework. It depends on the Lipschitz constant $c(\lambda)$. A comparable hypothesis can be found in the direct case in \cite{kolt}, except for the constant $c(\lambda)$. Section 4 is dedicated to the statement of \eqref{modulus}. Under standard complexity conditions, such as $L_2(\mu)$-entropy of the loss class $\{\ell(g),g\in\GG\}$, \eqref{modulus} holds true (see Lemma \ref{dudleynoisy} in Section 4 and the related discussion). It allows us to consider many examples of hypothesis spaces from finite VC classes to more complex functional classes such as kernel classes.

At this time, it is important to note that Theorem \ref{mainresult} depends on measure $\mu$ introduced in Definition \ref{LB} and \ref{KKb}. In the rest of the paper, we will consider two particular cases: $\mu=\nu\otimes P_Y$ ($\mu=\nu_Y$ for short in the sequel) and $\mu=P$. The Lipschitz property \textbf{(L$_\mu$)} with $\mu=P$ is stronger than \textbf{(L$_\mu$)} with $\mu=\nu\otimes P_Y$. Indeed, for any measurable function  $h:\mathcal{X}\times \mathcal{Y}\to\R$, if $\no f_y\no_\infty\leq C_y$, $\forall y\in\mathcal{Y}$:
$$
\E_Pf^2\leq \max_{y\in\mathcal{Y}}C_y\sum_{y\in\mathcal{Y}}p_y\int f(x,y)^2\nu(dx)= \max_{y\in\mathcal{Y}}C_y\,\Arrowvert f\Arrowvert^2_{L_2(\nu_Y)}.
$$
Since $\Arrowvert \cdot\Arrowvert_{L_2(P)}\leq C\Arrowvert \cdot \Arrowvert_{L_2(\nu_Y)}$ for some $C>0$, a Bernstein class with respect to $\nu_Y$ is also Bernstein with respect to $P$ (see Definition \ref{KKb}). The most favorable case ($\mu=\nu_Y$) arises in binary classification (see \cite{tsybakov2004} or \cite{nedelec}). Section \ref{application} states rates of convergence in these two different settings.

Finally, Theorem \ref{mainresult} requires the unicity of the Bayes $g^*$. Such a restriction can be avoided using a more sophisticated geometry as in \cite[Section 4]{kolt}.

\section{Applications}
\label{application}
In this section, we propose to apply the general upper bound of Theorem \ref{mainresult} to give rates of convergence of $\lambda$-ERM in two distinct frameworks. The first result deals with the errors-in-variables case where operator $A$ is a convolution product. Using kernel deconvolution estimators, we obtain fast rates of convergence. Then, we consider the general case using a family of projection estimators into the SVD basis of the operator. We also consider two different settings in the sequel, namely $\mu=\nu_Y$ and $\mu=P$ (see the discussion at the end of Section 2). In this case, we restrict the study to a compact set $K\subseteq \mathcal{X}$.
\subsection{Errors-in-variables case}
The elementary model of indirect observations is the additive measurement error model with known error density. In this case, we suppose that we observe a corrupted training set $(Z_i,Y_i),\,i=1,\ldots, n$ where:
$$
Z_i=X_i+\epsilon_i,\,i=1,\ldots ,n.
$$
The sequence of random variables $\epsilon_1,\ldots, \epsilon_n$ are i.i.d. $\R^d$-random variables with density $\eta$ with respect to the Lebesgue measure on $\R^d$. In this situation, operator $A$ is exactly known as a convolution product with density $\eta$. Note that in practical applications, this knowledge cannot be guaranteed. However, in most examples, we are able to estimate the error density $\eta$ from replicated measurements. In the sequel, we do not address this problem and we focus on the deconvolution step itself. \\In the errors-in-variables case, the difficulty of this inverse problem can be represented thanks to the asymptotic behavior of the Fourier transform of the noise density $\eta$. Assumption \textbf{(A1)} below concerns the asymptotic behavior of the characteristic function of the noise distribution. These kind of restrictions are standard in deconvolution problems (see \cite{Fan,butucea,meister}). 
 \\

\noindent
\textbf{(A1)} \textit{There exist $(\beta_1,\dots,\beta_d)'\in \R_+^d$ such that for all $i\in \lbrace 1,\dots, d \rbrace$, $\beta_i>\frac{1}{2}$ and:
$$ \left| \mathcal{F}[\eta_i](t) \right| \sim |t|^{-\beta_i}, \mathrm{as} \ t\to +\infty,$$
where $\mathcal{F}[\eta_i]$ denotes the Fourier transform of $\eta_i$. Moreover, we assume that $\mathcal{F}[\eta_i](t) \not = 0$ for all $t\in \R$ and $i\in \lbrace 1,\dots, d \rbrace$.\\}\\
Assumption \textbf{(A1)} focuses on moderately ill-posed inverse problems by considering polynomial decay of the Fourier transform. Notice that straightforward modifications in the proofs allow to consider severely ill-posed inverse problems.\\
In this framework,  we construct kernel deconvolution estimators of the densities $f_y,y\in\mathcal{Y}$. For this purpose, let us introduce $\mathcal{K}=\prod_{j=1}^d \mathcal{K}_j:\R^d \to \R$ a $d$-dimensional function defined as the product of $d$ unidimensional function $\mathcal{K}_j$. 
Then if we denote by $\lambda=(\lambda_1,\dots,\lambda_d)\in\R^d_+$ a set of (positive) bandwidths, we define $\mathcal{K}_\eta$ as
\begin{eqnarray}
\mathcal{K}_{\eta} & : & \R^d \to \R \nonumber \\
& & t \mapsto \mathcal{K}_\eta(t) = \FF^{-1}\left[ \frac{\FF[\mathcal{K}](\cdot)}{\FF[\eta](\cdot/\lambda)}\right](t).
\label{dk}
\end{eqnarray}
To apply Theorem \ref{mainresult}, we also need the following assumption on the regularity of the conditional densities:
\\

\noindent
\textbf{(R1)} \textit{Given $\gamma,L>0$, for any $y\in\mathcal{Y}$, $f_y\in\mathcal{H}(\gamma,L)$ where:
\beqnn
\mathcal{H}(\gamma,L)=\{f\in\Sigma(\gamma,L):f \mbox{ are bounded probability densities w.r.t. Lebesgue}\},
\eeqnn
and $\Sigma(\gamma,L)$ is the class of isotropic H\"older continuous functions $f$ having continuous partial derivatives up to order $\lfloor \gamma \rfloor$, the maximal integer strictly less than $\gamma$ and such that:
\beqnn
|f(y)-p_{f,x}(y)|\leq L|x-y|^\gamma,
\eeqnn
where $p_{f,x}$ is the Taylor polynomial of $f$ at order $\lfloor \gamma \rfloor$ at point $x$.
}\\

\hspace{-0.6cm} This H\"older regularity is standard to control the bias term of kernel estimators in density estimation or density deconvolution (see for instance \cite{booktsybakov}).

In this context, for all $g\in\mathcal{G}$, we define the $\lambda$-ERM \eqref{decerm} with empirical risk:
\beqn
\label{ldec}
R_n^\lambda(g)=\frac{1}{n}\sum_{i=1}^n\ell_\lambda(g,(Z_i,Y_i)),
\eeqn
where $\ell_\lambda(g,(z,y))$ is given by:
\beqnn
\ell_\lambda(g,(z,y))=\int_{\R^d} \ell(g(x),y)\frac{1}{\lambda}\mathcal{K}_\eta\left(\frac{z-x}{\lambda}\right)dx,
\eeqnn
where with a slight abuse of notations we write for any $z=(z_1,\ldots,z_d)$,\\ $x=(x_1,\ldots,x_d)\in\R^d$, $\lambda=(\lambda_1,\ldots,\lambda_d)\in\R^d_+$:
$$
\frac{1}{\lambda}\mathcal{K}_\eta\left(\frac{z-x}{\lambda}\right)=\Pi_{i=1}^d\frac{1}{\lambda_i}\mathcal{K}_\eta\left(\frac{z_1-x_1}{\lambda_1},\cdots,\frac{z_d-x_d}{\lambda_d}\right).
$$
Theorem \ref{deconv} below presents the rates of convergence of $\lambda$-ERM  under assumptions \textbf{(A1)-(R1)}.
\begin{thm}
\label{deconv}
Suppose $\{\ell(g)-\ell(g^*),g\in\mathcal{G}\}$ is a Bernstein class with respect to $\nu_Y$ with parameter $\kappa\geq 1$ and $\ell(g(\cdot),y)\in L_2(\R^d)$, for any $y\in\mathcal{Y}$. Suppose $0<\rho<1$ exists such that: 
\beqnn
\tilde{\omega}_n(\GG,\delta,\nu_Y)\leq C_1 \frac{c(\lambda)}{\sqrt{n}}\delta^{1-\rho},\forall 0<\delta<1,
\eeqnn
for some $C_1>0$.\\
Under \textbf{(A1)} and \textbf{(R1)}, we have, for $n$ great enough:
\beqnn 
\sup_{f_y\in\mathcal{H}(\gamma,L)}\E R_{\ell}(\hat{g})-R_{\ell}(g^*)\leq C n^{-\frac{\kappa\gamma}{\gamma(2\kappa+\rho-1)+(2\kappa-1)\bar{\beta}}},
\eeqnn
where $\bar{\beta}=\sum_{i=1}^d\beta_i$ and $ \lambda=(\lambda_1,\dots,\lambda_d)$ is given by:
\beqn
\label{deconvchoice}
\forall i\in\{1,\ldots,d\}, \,\lambda_i= n^{-\frac{2\kappa-1}{2\gamma(2\kappa+\rho-1)+2(2\kappa-1)\bar{\beta}}}.
\eeqn
\end{thm}
The proof of this result is postponed to Section 6. Here follows some remarks.\\

Rates in Theorem \ref{deconv} generalize the result of \cite{kolt} (see also \cite{tsybakov2004}) to the errors-in-variables case. Point out that if $\bar{\beta}=0$, we get the rates of the direct case. Here, the price to pay for the inverse problem of deconvolution can be quantified as $\frac{(2\kappa-1)\bar{\beta}}{\gamma},
$ where $\kappa>1$.
Hence, the performances of the method depend on the behavior of the characteristic function of the noise distribution. In pattern recognition, it is important to notice that the influence of the errors in variables is related to both parameters $\kappa$ and $\gamma$. Same phenomenon also occurs in \cite{pinkfloyds}.\\ 
It is also interesting to study the minimax optimality of the result of Theorem \ref{deconv} using the lower bounds presented in \cite{pinkfloyds}. For this purpose, let us introduce a random couple $(X,Y)$ with law $P$ on $\mathcal{X}\times\{0,1\}$. Given $\mathcal{G}$ and the class of associated candidates $\{g(x)=\ind_G(x),\,G\in\mathcal{G}\}$, we consider the hard loss $\ell_H(g(x),y)=|y-\ind_G(x)|$. In this case, the Bayes risk is defined as:
$$
R_H(G)=\E|Y-\ind_G(X)|.
$$
It is easy to see that for $y\in\{0,1\}$ and $g(x)=\ind_G(x)$, we have:
$$
|\ell_H(g(x),y)-\ell_H(g'(x),y)|=\left||y-\ind_G(x)|-|y-\ind_{G'}(x)|\right|=|\ind_G(x)-\ind_{G'}(x)|.
$$
Gathering with the margin assumption, Lemma 2 in \cite{mammen} allows us to write:
\beqnn
\Arrowvert \ell_H(g)-\ell_H(g')\Arrowvert^2_{L_2(\nu_Y)}=\Arrowvert \ind_G-\ind_{G'}\Arrowvert^2_{L_2(\R^d)}&=&d_\Delta(G,G')\\ 
&\leq &\frac{c_0}{2}\left(R_H(g)-R_H(g')\right)^{\frac{\alpha}{\alpha+1}}.
\eeqnn
As a result, provided that $G^*\in\mathcal{G}$ and under the margin assumption, the excess loss class $\{\ell_H(g)-\ell_H(g^*)\}$ is Bernstein with respect to $\mu=\nu_Y$ with parameter $\kappa=\frac{\alpha+1}{\alpha}$. \\
To apply Theorem \ref{deconv}, we need to check \textbf{(L$_\mu$)} and \textbf{(B)} from Definition \ref{LB}. Remark that from Lemma 3 in \cite{pinkfloyds}, we have:
$$
\no l_\lambda(g)-l_\lambda(g')\no^2_{L_2(\tilde{P})}\leq C\Pi_{i=1}^d \lambda_i^{-\beta_i}d_\Delta(G,G'),
$$
where for any  $g=\ind_G$:
$$
\ell_\lambda(g,z,y)=\int \ell_H(g(x),y)\frac{1}{\lambda}\mathcal{K}_\eta\left(\frac{z-x}{\lambda}\right)dx.
$$
Consequently, $\{l_\lambda(g),g=\ind_G:G\in\mathcal{G}\}$ is a LB-class with respect to $\nu_Y$ with constants $c(\lambda)$ and $K(\lambda)$ given by:
$$
c(\lambda)=\Pi_{i=1}^d\lambda_i^{-\beta_i} \mbox{ and } K(\lambda)=\Pi_{i=1}^d\lambda_i^{-\beta_i-1/2}.
$$
The last step is to control the complexity parameter $\tilde{\omega}_n(\mathcal{G},\delta,\nu_Y)$ as a function of $\delta$. With Lemma 5.1 in \cite{AT}, a control of the $L_2(\nu_Y)$-entropy with bracketing of the class $\{\ind_G,\,G\in\mathcal{G}\}$ is given by:
$$
\log\mathcal{N}(\{\ind_G,\,G\in\mathcal{G}\},L_2(\nu_Y),\epsilon)\leq c\epsilon^{-\frac{d}{\gamma\alpha}},
$$
under a plug-in type regularity assumption such as \textbf{(R1)}. As a result, we can apply Lemma \ref{dudleynoisy} in Section 4 to get a control of the desired modulus of continuity as follows:
$$
\tilde{\omega}_n(\GG,\delta,\nu_Y)\leq C_1 \frac{c(\lambda)}{\sqrt{n}}\delta^{1-\frac{d}{\gamma\alpha}},
$$
for some $C_1>0$.\\
Finally, using Lemma 4 in Section 6, in the particular case of the hard loss, $\{\ell_\lambda(g),g\in\mathcal{G}\}$ has approximation power $a(\lambda)$ with constant $0<r<1$ given by:
$$
a(\lambda)=\sum_{i=1}^d\lambda_i^{\frac{\kappa}{\kappa-1}\gamma}\mbox{ and }r=\frac{1}{\kappa}.
$$
In this case, Theorem \ref{deconv} leads to:
\beqnn 
\E R_{H}(\hat{g}_n^\lambda)-R_{H}(g^*)\leq C n^{-\frac{(\alpha+1)\gamma}{\gamma(\alpha+2)+d+2\bar{\beta}}}.
\eeqnn
This rate corresponds to the minimax rates of classification with errors in variables stated in \cite{pinkfloyds}. It ensures the minimax optimality of the method in the errors-in-variables case for this particular loss. An open problem is to give a lower bound for more general losses.
\subsection{General case with singular values decomposition} 
In this section, we observe a training set $(Z_i,Y_i),\,i=1,\ldots, n$ where
$Z_i$ are i.i.d. with law $Af$, where $A:L_2(\mathcal{X})\to L_2(\tilde{\mathcal{X}})$ is a known linear compact operator. For simplicity, we also restrict ourselves to moderately ill-posed inverse problem considering the singular values decomposition of $A$. Since $A$ is compact, $A^*A$ is auto-adjoint and compact. We can find an orthonormal basis of eigenfunctions of $A^*A$, denoted by $(\phi_k)_{k\in\N^*}$. We obtain $A^*A\phi_k=b_k^2\phi_k$, with $(b_k)_{k\in\N^*}$ the decreasing sequence of singular values. Considering the image basis $\psi_k=A\phi_k/b_k$, we have the following SVD (singular values decomposition):
\beqn
\label{svd}
A\phi_k=b_k\psi_k\mbox{ and }A^*\psi_k=b_k\phi_k,\,k\in\N^*.
\eeqn
In the sequel, we make the following assumption:\\
\noindent
\textbf{(A2)} \textit{There exists $\beta\in \R_+$ such that:
$$ b_k \sim k^{-\beta} \mathrm{as} \ k\to +\infty.$$}
In this case, the rate of decrease of the singular values is polynomial. As an example, we can consider the convolution operator above and from an easy calculation, the spectral domain is the Fourier domain and \textbf{(A2)} is comparable to \textbf{(A1)}. However assumption \textbf{(A2)} can deal with any linear inverse problem and is rather standard in the statistical inverse problem literature (see \cite{cavaliersurvey}).

In this framework,  we also need the following assumption on the regularity of the conditional densities into the basis of the operator $A$:\\

\noindent
\textbf{(R2)} \textit{For any $y\in\mathcal{Y}$, $f_y\in\mathcal{P}(\gamma,L)$ where:
\beqnn
\mathcal{P}(\gamma,L)&=&\{f\in\Theta(\gamma,L):f \mbox{ are bounded probability densities w.r.t. Lebesgue }\},
\eeqnn
and $\Theta(\gamma,L)$ is the ellipso\"id in the SVD basis defined as:
\beqnn
\Theta(\gamma,L)=\{f(x)=\sum_{k\geq 1}\theta_k\phi_k(x):\sum_{k\geq 1}\theta_k^2k^{2\gamma}\leq L\}.
\eeqnn
}
Considering the SVD \eqref{svd}, we propose to replace in the true risk the conditional densities $f_y$ by a family of projection estimators given by:
\beqn
\label{projection}
\hat{f}_y(x)=\sum_{k= 1}^N\hat{\theta}_k^y\phi_k(x),
\eeqn
where $\hat{\theta}_k^y$ is an unbiased estimator of $\theta_k^y=\int f_y\phi_kd\nu$ given by:
\beqn
\hat{\theta}_k^y=\frac{1}{n_y}\sum_{i=1}^{n_y}b_k^{-1}\phi_k(Z_i).
\eeqn
In this case, assumption \eqref{est} is satisfied with $k_N(z,x)=\sum_{k=1}^Nb_k^{-1}\phi_k(z)\phi_k(x)$. It gives the following expression of the empirical risk:
\beqnn
R_n^N(g)=\frac{1}{n}\sum_{i=1}^n\ell_N(g,Z_i,Y_i),
\eeqnn
where:
$$
\ell_N(g,z,y)=\sum_{k=1}^{N}b_k^{-1}\int_\mathcal{X}\phi_k(x)\ell(g(x),y)\nu(dx)\phi_k(z).
$$
Next theorem states the rates of convergence for the ERM estimator $\hat{g}_n^N$ defined as:
\beqnn
\hat{g}_n^N=\arg\min_{g\in\mathcal{G}}\frac{1}{n}\sum_{i=1}^n\ell_N(g,Z_i,Y_i).
\eeqnn
\begin{thm}
\label{svdtheo}
Suppose $\{\ell(g)-\ell(g^*),g\in\mathcal{G}\}$ is Bernstein class with respect to $\nu_Y$ with parameter $\kappa\geq 1$ such that $\ell(g(\cdot),y)\in L_2(\nu)$, for any $y\in\mathcal{Y}$. Suppose $0<\rho<1$ exists such that:
\beqnn
\tilde{\omega}_n(\GG,\delta,\nu_Y)\leq C_1\frac{c(N)}{\sqrt{n}}\delta^{1-\rho},\,\forall 0<\delta<1,
\eeqnn
for some $C_1>0$.
Then under \textbf{(A2)} and \textbf{(R2)}, $\hat{g}_n^N$ satisfies, for $n$ great enough:
\beqnn 
\sup_{f_y\in\mathcal{P}(\gamma,l)}\E R_\ell(\hat{g}_n^N)-R_\ell(g^*)\leq C n^{-\frac{\kappa\gamma}{\gamma(2\kappa+\rho-1)+(2\kappa-1)\beta}},
\eeqnn
where we choose $N$ such that:
$$
N= n^{\frac{2\kappa-1}{2\gamma(2\kappa+\rho-1)+2(2\kappa-1)\beta}}.
$$
\end{thm}
Theorem \ref{svdtheo} shows that in pattern recognition with indirect observations, we can deal with any linear compact operator $A$ using the SVD. From this point of view, this result could be compared with \cite{klemela} where white noise model is considered.\\
Rates of convergence in Theorem \ref{svdtheo} are comparable with Theorem 2. If $A$ is a convolution operator, the result above shows that $\hat{g}_n^N$ using projection estimators in the SVD reaches the rate of Theorem \ref{deconv} using kernel deconvolution estimators. In this case, the regularity assumption deals with ellipsoids in the SVD domain instead of H\"older classes. However, we can conjecture that this result is also minimax, although a rigorous lower bound has to be managed.\\
 Finally, this result might be extended to other linear regularization methods without significant change. Here, we present the result for projections into the SVD domain for the sake of simplicity in the proofs but Tikhonov and Landweber regularization could be considered for instance.
\subsection{Restriction to a compact $K$}
\label{appendix}
In this subsection, we develop an alternative to Theorem \ref{deconv}-\ref{svdtheo} to deal with a weaker Bernstein assumption. For the sake of simplicity, we restrict ourselves in Theorem \ref{deconv}-\ref{svdtheo} to Bernstein class with respect to measure $\nu_Y=\nu\otimes P_Y$ (see Definition \ref{KKb}). In this case, it is sufficient to deal with LB-class with respect to $\nu_Y$ in Definition \ref{LB}, thanks to \eqref{serie}. However, Bernstein classes with respect to $\nu_Y$ appear only in particular case, such as classification with hard loss in the context of \cite{mammen,tsybakov2004} (see Section 3.1). Here, we present a corollary of Theorem \ref{deconv}-\ref{svdtheo}. It allows us to deal with Bernstein classes in the spirit of \cite{empimini}, namely such that:
$$
\E_Pf^2\leq \kappa_0\left(\E_Pf\right)^{1/\kappa},\,\forall f\in\mathcal{F}=\{\ell(g)-\ell(g^*),g\in\mathcal{G}\}.
$$
\\
The idea is to restrict the study to a set $K\subseteq\R^d$ where $f\geq c_0>0$ over $K$. For this purpose, we can consider a set $\mathcal{G}$ of classifiers $g$ such that $\{x\in\mathcal{X}:f(x)>0\}\subset K$. We can also introduce the following loss:
\beqn
\label{restrictloss}
\ell_{\lambda,K}(g,z,y)=\int_Kk_\lambda(z,x)\ell(g(x),y)\nu(dx).
\eeqn
It means that we deal with the minimization of a true risk of the form:
$$
R_{\ell,K}(g)=\sum_{y\in\mathcal{Y}} p(y)\int_K \ell(g(x),y)f_y(x)dx.
$$
With \eqref{restrictloss}, it is straightforward to get \textbf{(L$_\mu$)} with $\mu=P$ since if $f\geq c_0>0$ on $K$, one gets:
\beqnn
\sum_{y\in\mathcal{Y}}p_y\int_K (\ell(g(x),y)-\ell(g'(x),y))^2\nu(dx)\leq\frac{1}{c_0}\no \ell(g)-\ell(g')\no_{L_2(P)}.
\eeqnn
Roughly speaking, Assumption \textbf{(L$_\mu$)} in Definition \ref{LB} whith $\mu=P$ provides a control of the variance of $\ell_\lambda(g,(Z,Y))$ by the variance of $\ell(g(X),Y)$. To have a control of the $L_2(\tilde{P})$-norm with respect to the $L_2(P)$-norm, we need to restrict the problem to $\{x:f(x)>0\}$. Otherwise, the variance of $\ell_\lambda(g,(Z,Y))$ cannot be compared with the variance of $\ell(g(X),Y)$. 

The following corollary points out the same performances for the $\lambda$-ERM over $K$ defined as:
$$
\hat{g}_n^{\lambda,K}=\arg\min_{g\in\mathcal{G}}\sum_{i=1}^n\ell_{\lambda,K}(g,Z_i,Y_i).
$$
\begin{cor}
Suppose $\{\ell(g)-\ell(g^*),g\in\mathcal{G}\}$ is a Bernstein class with respect to $P$ with parameter $\kappa>1$ and $\ell(g(\cdot),y)\in L_2(\nu)$, for any $y\in\mathcal{Y}$. Suppose $0<\rho<1$ exists such that: 
\beqnn
\tilde{\omega}_n(\GG,\delta,P)\leq C_1 \frac{c(\lambda)}{\sqrt{n}}\delta^{1-\rho},\,\forall 0<\delta<1.
\eeqnn
\begin{enumerate}
\item
Under \textbf{(A1)} and \textbf{(R1)}, $\hat{g}_n^{\lambda,K}$  satisfies, for $n$ great enough:
\beqnn 
\sup_{f_y\in\mathcal{H}(\gamma,l)}\E R_{\ell,K}(\hat{g}_n^{\lambda,K})-R_{\ell,K}(g^*)\leq C n^{-\frac{\kappa\gamma}{\gamma(2\kappa+\rho-1)+(2\kappa-1)\bar{\beta}}},
\eeqnn
where $\bar{\beta}=\sum_{i=1}^d\beta_i$ and for a choice of $ \lambda=(\lambda_1,\dots,\lambda_d)$ given by:
\beqn
\label{deconvchoice}
\forall i\in\{1,\ldots,d\}, \,\lambda_i= n^{-\frac{2\kappa-1}{2\gamma(2\kappa+\rho-1)+2(2\kappa-1)\bar{\beta}}}.
\eeqn
\item Under  \textbf{(A2)} and \textbf{(R2)}, $\hat{g}_n^{N,K}$ satisfies, for $n$ great enough:
\beqnn 
\sup_{f_y\in\mathcal{P}(\gamma,L)}\E R_{\ell,K}(\hat{g}_n^{N,K})-R_{\ell,K}(g^*)\leq C n^{-\frac{\kappa\gamma}{\gamma(2\kappa+\rho-1)+(2\kappa-1)\beta}}),
\eeqnn
where we choose $N$ such that:
$$
N=n^{\frac{2\kappa-1}{2\gamma(2\kappa+\rho-1)+2(2\kappa-1)\beta}}.
$$
\end{enumerate}
\end{cor}
This corollary allows to get the same fast rates of convergence of Theorem \ref{deconv}-\ref{svdtheo} under a weaker Bernstein assumption. The price to pay for the $\lambda$-ERM with restricted loss \eqref{restrictloss} relies on the dependence on $K$ of the estimation procedure. 
\section{Complexity from indirect observations}
\label{imc}
The main results of this paper rely on the control of the indirect modulus of continuity defined as:
$$
\tilde{\omega}_n(\GG,\delta,\mu)=\E\sup_{g,g'\in\mathcal{G}:\no \ell(g)-\ell(g')\no_{L_2(\mu)}\leq\delta}|\tilde{P}-\tilde{P}_n|(\ell_\lambda(g)-\ell_\lambda(g')).
$$
In this section, we intent to upper bound this quantity thanks to standard learning theory arguments. The first result links the control of $\tilde{\omega}_n(\mathcal{G},\delta,\mu)$ to the bracketing entropy of the loss class, which generalizes the result of the direct case (see \cite{wvdv}) when $A=Id$.
\begin{lemma}
\label{dudleynoisy}
Consider a LB-class $\{\ell_\lambda(g),g\in\GG\}$ with respect to $\mu$ with Lipschitz constant $c(\lambda)$.
Then, given some $0<\rho<1$, we have:
\beqnn
\mathcal{H}_{B}(\{\ell(g),\,g\in\mathcal{G}\},\epsilon,L_2(\mu))\leq c\epsilon^{-2\rho}\Rightarrow\tilde{\omega}_n(\GG,\delta,\mu)\leq C_1 \frac{c(\lambda)}{\sqrt{n}}\delta^{1-\rho},
\eeqnn
where $\mathcal{H}_{B}(\{\ell(g),\,g\in\mathcal{G}\},\epsilon,L_2(\mu))$ denotes the $\epsilon$-entropy with bracketing of the set $\{\ell(g),\,g\in\mathcal{G}\}$ with respect to $L_2(\mu)$ (see \cite{wvdv} for a definition).
\end{lemma}
With such a Lemma, it is possible to control the complexity in the indirect setup thanks to standard entropy conditions. The proof is presented in Section \ref{proofs}. It is based on a maximal inequality due to \cite{wvdv} applied to the class:
$$
\mathcal{F}_\lambda=\{\ell_\lambda(g)-\ell_\lambda(g'),g,g'\in\mathcal{G}:\no \ell (g)-\ell(g')\no_\mu\leq \delta\}.
$$

For instance, let us consider a loss $\ell$ such that $t\mapsto \ell(y,t)$ is a convex function, for any $y\in\mathcal{Y}$. Both least squares or large margin classification can be viewed as special cases of convex losses where $l(y,t)=(y-t)^2$ or $l(y,t)=\Phi(yt)$ respectively, with a given convex function $\Phi$ (such as $\Phi(u)=(1-u)_+$ for the hinge loss). In this case, using the convexity of the loss, it is straightforward to obtain with Lemma \ref{dudleynoisy} the following corollary.

\begin{cor}
\label{convex}
Suppose $\ell(\cdot,y)$ is convex for any $y\in\mathcal{Y}$, $\{\ell(g)-\ell(g^*),g\in\mathcal{G}\}$ is a Bernstein class with respect to $\mu$ with parameter $\kappa\geq 1$ and $\ell(g(\cdot),y)\in L_2(\nu)$, for any $y\in\mathcal{Y}$. Suppose $0<\rho<1$ exists such that:
\beqnn
\mathcal{H}_{B}(\mathcal{G},\epsilon,L_2(\mu))\leq c\epsilon^{-2\rho}, \,\forall 0<\epsilon<1.
\eeqnn
\begin{enumerate}
\item
Under \textbf{(A1)} and \textbf{(R1)}, the solution $\hat{g}_n^\lambda$ of the minimization \eqref{lerm} satisfies:
\beqnn 
\sup_{f_y\in\mathcal{H}(\gamma,l)}\E R_{\ell}(\hat{g}_n^\lambda)-R_{\ell}(g^*)\leq C n^{-\frac{\kappa\gamma}{\gamma(2\kappa+\rho-1)+(2\kappa-1)\bar{\beta}}},
\eeqnn
where $\bar{\beta}=\sum_{i=1}^d\beta_i$ and $ \lambda=(\lambda_1,\dots,\lambda_d)$ is given by:
\beqn
\label{deconvchoice}
\forall i\in\{1,\ldots,d\}, \,\lambda_i= n^{-\frac{2\kappa-1}{2\gamma(2\kappa+\rho-1)+2(2\kappa-1)\bar{\beta}}}.
\eeqn
\item Under \textbf{(A2)} and \textbf{(R2)}, $\hat{g}_n^N$ satisfies:
\beqnn 
\sup_{f_y\in\mathcal{P}(\gamma,L)}\E R_{\ell}(\hat{g}_n^N)-R_{\ell}(g^*)\leq C n^{-\frac{\kappa\gamma}{\gamma(2\kappa+\rho-1)+(2\kappa-1)\beta}},
\eeqnn
where we choose $N$ such that:
$$
N = n^{\frac{2\kappa-1}{2\gamma(2\kappa+\rho-1)+2(2\kappa-1)\beta}}.
$$
\end{enumerate}
\end{cor}
This corollary is a special version of somewhat more general analysis of the previous sections. It allows to consider standard hypothesis sets $\mathcal{G}$ such as VC classes or kernel classes (\cite{nedelec} or \cite{mendelsonkernel}).
Another possible powerful direction is to study directly the complexity of the class $\{\ell_\lambda(g),g\in\mathcal{G}\}$ thanks to entropy numbers of compact operators. For this purpose, note that if $\mathcal{X}$ is compact,
$
\ell_\lambda(g,z,y)=\int_\mathcal{X}k_\lambda(z,x)\ell(g(x),y)\nu(dx)
$ can be considered as the image of $\ell(g)$ by the integral operator $L_{k_\lambda}$ associated to the function $k_\lambda$. Hence we have:
$$
\{\ell_\lambda(g),g\in\mathcal{G}\}=L_{k_\lambda}(\{\ell(g),g\in\mathcal{G}\}).
$$
Furthermore, it is clear that if $k_\lambda$ is continuous, $L_{k_\lambda}$ is well-defined and compact. Using for instance \cite{will}, and provided that $\ell$ is bounded and $\mathcal{G}$ consists of bounded functions in $L_2(\nu,\mathcal{X})$, entropy of the class $\{\ell_\lambda(g),g\in\mathcal{G}\}$ could be controlled in terms of the eigenvalues of the integral operator. In this case, it is clear that the entropy of the class $\{\ell_\lambda(g),g\in\mathcal{G}\}$ depends strongly on the spectrum of the operator $A$.\\
More precisely, if $A$ is a convolution product, Section 3.1 deals with kernel deconvolution estimators. As a result, operator $L_{k_\lambda}$ is defined as the convolution product $L_{k_\lambda}f(z)=\frac{1}{\lambda}\mathcal{K}_{\eta}(\frac{\cdot}{\lambda})*f(z)$. 
Its spectrum is related to the behavior of the Fourier transform of the deconvolution kernel estimator, which corresponds to the quantity $\frac{\mathcal{F}[\mathcal{K}]}{\mathcal{F}[\eta](\frac{\cdot}{\lambda})}$. At the end, the control of the entropy of the class of interest $\{\ell_\lambda(g),g\in\mathcal{G}\}$ could be calculated thanks to an assumption over the behavior of the Fourier transform of the noise distribution $\eta$ such as \textbf{(A1)}.

\section{Conclusion}

This paper has tried to investigate the effect of indirect observations into the statement of fast rates of convergence in empirical risk minimization. Many issues could be considered in future works.

The main result is a general upper bound in the statistical learning context, when we observe indirect inputs $Z_i,\,i=1,\ldots ,n$ with law $Af$. The proof is based on a deviation inequality for supprema of empirical processes. It seems to fit the indirect case provided that it is used carefully. For this purpose, we introduce Lipschitz and bounded classes $\{\ell_\lambda(g),g\in\mathcal{G}\}$, depending on a smoothing parameter $\lambda$. It allows us to quantify the effect of the inverse problem on the empirical process machinery. The price to pay is summarized in a constant $c(\lambda)$ which exploses as $\lambda\to 0$. The behavior of this constant is related to the degree of ill-posedness. Here in the midly ill-posed case, $c(\lambda)$ grows polyniomally as a function of $\lambda$. 

The result of Section \ref{sec2} suggests the same degree of generality as the results of \cite{kolt} in the direct case. It is well-known that the work of Koltchinskii allows to recover most of the recent results in statistical learning theory and the area of fast rates. Consequently, there is a nice hope that many problems dealing with indirect observations could be managed following the guiding thread of this paper.\\
 
The estimation procedure proposed in this paper can be discussed for several reasons. Firstly, it is not adaptive in many sense. At the first glance, we can see three levels of adaptation: (1) adaptation to the operator $A$; (2) adaptation to the tunable parameter $\lambda$; (3) adaptation or model selection of the hypothesis space $\mathcal{G}$. At this time, it is important to note that at least in the direct case, the same machinery used to analyzed the order of the excess risk can be applied to produce penalized empirical risk  minimization (see \cite{tsybakov2005,kolt,svm,loustau2}). However, the construction of adaptive versions of $\lambda$-ERM of the previous sections is a challenging open problem.\\

Finally, the aim of this contribution was to derive excess risk bounds under standard assumptions over the complexity and the geometry of the considered class $\mathcal{G}$. An alternative point of view would be to state oracle-type inequalities. Indeed, Theorem \ref{mainresult}-\ref{svdtheo} could be written in terms of exact asymptotic oracle inequalities of the form:
$$
\E R_l(\hat{g}_n^\lambda)\leq \inf_{g\in\mathcal{G}}R_l(g)+r_n(\mathcal{G}),
$$
where the residual term $r_n(\mathcal{G})$ corresponds to the rates of convergence in Theorem  \ref{mainresult}-\ref{svdtheo}. In this setting, it is well-known that ERM estimators reach optimal fast rates under a Bernstein assumption. However, the Bernstein assumption presented in Definition \ref{KKb} is a strong assumption related to the geometry of the class $\mathcal{G}$. \cite{nonexactlecue} proposes to relax significantly the Bernstein assumption and point out non-exact oracle inequalities of the form:
$$
\E R_l(\hat{g}_n^\lambda)\leq (1+\epsilon)\inf_{g\in\mathcal{G}}R_l(g)+r_n(\mathcal{G}),
$$
for some $\epsilon>0$. These results hold without Bernstein condition for any non-negative loss functions. There is a nice hope that such a study can be done in the presence of indirect observations, using some minor modifications in the proofs.

\section{Proofs}
\label{proofs}
The main ingredient of the proofs is a concentration inequality for empirical processes in the spirit of Talagrand (\cite{talagrand}). We use precisely a Bennet deviation bound for suprema of empirical processes due to Bousquet (see \cite{bousquet}) applied to a class of measurable functions $f\in\mathcal{F}$ from $\mathcal{X}$ into $[0,K]$. In this case it is stated in \cite{bousquet} that for all $t>0$:
$$
\P\left(Z\geq \E Z+\sqrt{2t(n\sigma^2+(1+K)\E Z)}+\frac{t}{3} \right)\leq\exp(-t),
$$
where
$$
Z=\sup_{f\in\FF}\left|\sum_{i=1}^n f(X_i)\right| \mbox{ and } \sup_{f\in\FF} \mathrm{Var}(f(X_1))\leq\sigma^2.
$$
The proof of Lemma \ref{interm} below uses iteratively Bousquet's inequality and gives rise to solve the fixed point equation as in \cite{kolt}. For this purpose, we introduce, for a function $\psi:\R_+\to\R_+$, the following transformations:
$$
\breve{\psi}(\delta)=\sup_{\sigma\geq \delta}\frac{\psi(\sigma)}{\sigma}\mbox{ and }\psi^\dagger(\epsilon)=\inf\{\delta>0:\breve{\psi}(\delta)\leq \epsilon\}.
$$
We are also interested in the following discretization version of these transformations:
$$
\breve{\psi}_q(\delta)=\sup_{\delta_j\geq \delta}\frac{\psi(\delta_j)}{\delta_j}\mbox{ and }\psi^\dagger_{q}(\epsilon)=\inf\{\delta>0:\breve{\psi}_{q}(\delta)\leq \epsilon\},
$$
where for some $q>1$, $\delta_j=q^{-j}$ for $j\in\mathbb{N}$.

In the sequel, constant $K,C>0$ denote generic constants that may vary from line to line. 
\subsection{Proof of Theorem \ref{mainresult}}
\begin{lemma}
\label{interm}
Suppose $\{\ell_\lambda(g),g\in\mathcal{G}\}$ is such that:
$$
\sup_{g\in\mathcal{G}}\no \ell_\lambda(g)\no_\infty\leq K(\lambda).
$$
Suppose $\{\ell_\lambda(g),g\in\mathcal{G}\}$ has approximation function $a(\lambda)$ and residual constant $0<r<1$ according to Definition \ref{app}. Define, for some constant $K>0$:
\beqnn
U_{n}^\lambda(\delta_j,t)=K\left[\phi_n^\lambda(\GG,\delta_j)+\sqrt{\frac{t}{n}}D^\lambda(\GG,\delta_j)+\sqrt{\frac{t}{n}(1+K(\lambda))\phi_n^\lambda(\GG,\delta_j)}+\frac{t}{n}\right],
\eeqnn
\beqnn
\phi_n^\lambda(\GG,\delta_j)=\E \sup_{g,g'\in\GG (\delta_j)}|\tilde{P}_n-\tilde{P}|[\ell_\lambda(g)-\ell_\lambda(g')],
\eeqnn
\beqnn
D^\lambda(\GG,\delta_j)=\sup_{g,g'\in\GG (\delta_j)}\sqrt{\tilde{P}(\ell_\lambda(g)-\ell_\lambda(g'))^2}.
\eeqnn
Then $\forall \delta\geq \delta_n^\lambda(t)=[U_{n}^\lambda(\cdot,t)]^\dagger_q(\frac{1-r}{2q})$, if $a(\lambda)\leq\frac{1-r}{4q}\delta$ we have for $\hat{g}=\hat{g}_n^\lambda$:
\beqnn
\mathbb{P}(R_{\ell}(\hat{g})-R_{\ell}(g^*)\geq \delta)\leq\log_q(\frac{1}{\delta}) e^{-t}.
\eeqnn
\end{lemma}
\begin{proof} The proof follows \cite{kolt} extended to the noisy set-up.

Given $q>1$, we introduce a sequence of positive numbers:
\beqnn
\delta_j=q^{-j},\,\forall j\geq 1.
\eeqnn
Given $n,j\geq 1$, $t>0$ and $\lambda\in\R^d_+$, consider the event:
\beqnn
E_{n,j}^\lambda(t)=\left\{\sup_{g,g'\in\GG (\delta_j)}|\tilde{P}_n-\tilde{P}|[\ell_\lambda(g)-\ell_\lambda(g')]\leq U_{n}^\lambda(\delta_j,t)\right\}.
\eeqnn
Then, we have, using Bousquet's version of Talagrand's concentration inequality (see \cite{bousquet}), for some $K>0$, $\P(E_{n,j}^{\lambda}(t)^C)\leq e^{-t}$, $\forall t\geq 0$.\\
We restrict ourselves to the event $E_{n,j}^\lambda(t)$. \\
Using Definition \ref{app}, we have with a slight abuse of notations:
\beqnn
R_{\ell}(\hat{g})-R_{\ell}(g^*)&\leq& (\tilde{P}_n-\tilde{P})(\ell_\lambda(g^*)-\ell_\lambda(\hat{g}))+(R_{\ell}-R^\lambda_\ell)(\hat{g}-g^*)\\
&\leq&(\tilde{P}_n-\tilde{P})(\ell_\lambda(g^*)-\ell_\lambda(\hat{g}))+a(\lambda)+r(R_{\ell}(\hat{g})-R_{\ell}(g^*)).
\eeqnn
Hence, we have:
\beqnn
\delta_{j+1}\leq R_{\ell}(\hat{g})-R_{\ell}(g^*)\leq \delta_j\Rightarrow \delta_{j+1}\leq \frac{1}{1-r}\left((\tilde{P}_n-\tilde{P})(\ell_\lambda(g^*)-\ell_\lambda(\hat{g}))+a(\lambda)\right).
\eeqnn
On the event $E_{n,j}^\lambda(t)$, it follows that $\forall \delta\leq\delta_j$:
\beqnn
\delta_{j+1}\leq R_{\ell}(\hat{g})-R_{\ell}(g^*)\leq \delta_j\Rightarrow\delta_{j+1}&\leq& \frac{1}{1-r}U_{n}^\lambda(\delta_j,t)+\frac{1}{1-r}a(\lambda)\\
&\leq &  \frac{\delta_j}{1-r}V_n^\lambda(\delta,t)+\frac{1}{1-r}a(\lambda),
\eeqnn
where $V_n^\lambda(\delta,t)=\breve{U}_{n}^\lambda(\delta,t)$ satisfies (see \cite{kolt}):
$$
U_{n}^\lambda(\delta_j,t)\leq \delta_jV_n^\lambda(\delta,t), \forall\delta\leq \delta_j. 
$$
We obtain:
\beqnn
\frac{1}{1-r}V_{n}^\lambda(\delta,t)\geq \frac{1}{q}-\frac{q^j}{1-r}a(\lambda)> \frac{1}{2q},
\eeqnn
since we have:
$$
a(\lambda)\leq\frac{1-r}{4q}\delta\Longrightarrow \frac{q^j}{1-r}a(\lambda)<\frac{1}{2q}.
$$
It follows from the definition of the $\dagger$-transform that:
\beqnn
\delta< [{U}_{n}^\lambda(\cdot,t)]^{\dagger}(\frac{1-r}{2q})=\delta_n^\lambda(t).
\eeqnn
Hence, we have on the event $E_{n,j}^\lambda(t)$, for $\delta_j\geq \delta$:
\beqnn
\delta_{j+1}\leq R_{\ell}(\hat{g})-R_{\ell}(g^*)\leq \delta_j\Rightarrow\delta\leq \delta_n^\lambda(t),
\eeqnn
or equivalently,  
\beqnn
\delta_n^\lambda(t)\leq \delta\leq \delta_j\Rightarrow \hat{g}\notin\GG(\delta_j,\delta_{j+1}),
\eeqnn

where $\GG(c,C)=\{g\in\mathcal{G}:c\leq R_{\ell}(g)-R_{\ell}(g^*)\leq C\}$. We eventually obtain:
\beqnn
\bigcap_{\delta_j\geq \delta}E_{n,j}^\lambda(t)\mbox{ and }\delta\geq \delta_n^\lambda(t)\Rightarrow R_{\ell}(\hat{g})-R_{\ell}(g^*)\leq \delta.
\eeqnn
This formulation allows us to write by union's bound:
\beqnn
\mathbb{P}(R_\ell(\hat{g})-R_\ell(g^*)\geq \delta)\leq \sum_{\delta_j\geq \delta} \P(E_{n,j}^{\lambda }(t)^C)\leq \log_q\left(\frac{1}{\delta}\right)e^{-t},
\eeqnn
since $\{j:\delta_j\geq \delta\}=\{j:j\leq -\frac{\log \delta}{\log q}\}$.
\end{proof}
\begin{proof}[Proof of Theorem 1]
The proof is a direct application of Lemma 1. We have, for some constant $K>0$:
\beqnn
U_n^\lambda(\delta,t)= K\left[\phi_n^\lambda(\GG,\delta)+\sqrt{\frac{t}{n}\phi_n^\lambda(\GG,\delta)(1+K(\lambda))}+\sqrt{\frac{t}{n}}D^\lambda(\GG,\delta)+\frac{t}{n}\right].
\eeqnn
Using the Bernstein condition gathering with the complexity assumption over $\tilde{\omega}_n(\GG,\delta)$, we have:
\beqnn
\phi_n^\lambda(\GG,\delta)&\leq &\E \sup_{g,g'\in\GG (\delta)}|\tilde{P}_n-\tilde{P}|[\ell_\lambda(g)-\ell_\lambda(g')]\\
&\leq& \E \sup_{g,g'\in\GG: \no \ell(g)-\ell(g')\no_{L^2(\mu)}\leq 2\sqrt{\kappa_0}\delta^{\frac{1}{2\kappa}}}|\tilde{P}_n-\tilde{P}|[\ell_\lambda(g)-\ell_\lambda(g')]=\tilde{\omega}_n(\GG,2\sqrt{\kappa_0}\delta^{\frac{1}{2\kappa}})\\
&\leq & C\frac{c(\lambda)}{\sqrt{n}}\delta^{\frac{1-\rho}{2\kappa}}.
\eeqnn
A control of $D^\lambda(\GG,\delta)$ using the Lipschitz assumption leads to:
\beqnn
U_n^\lambda(\delta,t)\leq C\left[\frac{c(\lambda)}{\sqrt{n}}\delta^{\frac{(1-\rho)}{2\kappa}}+\frac{c(\lambda)^{1/2}}{n^{3/4}}\delta^{\frac{1-\rho}{4\kappa}}\sqrt{K(\lambda)t}+\sqrt{\frac{t}{n}} c(\lambda)\delta^{\frac{1}{2\kappa}}+\frac{t}{n}\right].
\eeqnn
Hence we have from an easy calculation:
\beqnn
\delta_n^\lambda(t)\leq C \max\left(\left(\frac{c(\lambda)}{\sqrt{n}}\right)^{\frac{2\kappa}{2\kappa+\rho -1}},\frac{[c(\lambda)K(\lambda)]^{\frac{2\kappa}{4\kappa+\rho-1}}}{n^{\frac{3\kappa}{4\kappa+\rho-1}}}t^{\frac{2\kappa}{4\kappa+\rho-1}},\left(\frac{c(\lambda)}{\sqrt{n}}\right)^{\frac{2\kappa}{2\kappa-1}}t^{\frac{2\kappa}{2\kappa -1}},\frac{t}{n}\right).
\eeqnn
Consequently, for any $0<t\leq 1$, for $n$ large enough, we have:
$$
\left(\frac{c(\lambda)}{\sqrt{n}}\right)^{\frac{2\kappa}{2\kappa+\rho -1}}\geq \delta_n^\lambda(t+\log\log_q n),
$$
provided that:
$$
K(\lambda)\leq \frac{c(\lambda)^{\frac{2\kappa}{2\kappa+\rho-1}}n^{\frac{\kappa+\rho-1}{2\kappa+\rho-1}}}{1+\log\log_q n}.
$$
It remains to use Lemma 2 with $t$ replaced by $t+\log\log_q n$ to obtain:
$$
\P\left(R_\ell(\hat{g}_n^\lambda)-R_\ell(g^*)\geq K(1+t)\left(\frac{c(\lambda)}{\sqrt{n}}\right)^{\frac{2\kappa}{2\kappa+\rho -1}}\right)\leq e^{-t},
$$
provided that the approximation function obeys to the following inequality:
\beqnn
a(\lambda)\leq K\frac{(1-r)}{4q}\left(\frac{c(\lambda)}{\sqrt{n}}\right)^{\frac{2\kappa}{2\kappa+\rho -1}}.
\eeqnn
\end{proof}
\subsection{Proof of Theorem \ref{deconv}}
Theorem 2 is a straightforward application of Theorem 1 to the particular case of errors in variables using deconvolution kernel estimators. 

First step is to check that the estimation procedure described in Section 3.1 gives rise to a LB-class with respect to $\nu_Y$ where $\nu$ is the Lebesgue measure on $\R^d$.
\begin{lemma}
\label{liphold}
Suppose \textbf{(A1)} holds and suppose $l(g(\cdot),y)\in L_2(\mathcal{X})$ for any $y\in\mathcal{Y}$. Consider a deconvolution kernel $\mathcal{K}_\eta(t) = \FF^{-1}\left[ \frac{\FF[\mathcal{K}](\cdot)}{\FF[\eta](\cdot/\lambda)}\right]$ where $\mathcal{K}(t)=\Pi_{i=1}^d\mathcal{K}_i(t_i)$ where $\mathcal{K}_i$ have compactly supported and bounded Fourier transform. Then we have:
\beqnn
\no \ell_\lambda(g)-\ell_\lambda(g')\no_{L_2(\tilde{P}}\lesssim \Pi_{i=1}^d\lambda_i^{-\beta_i}\no \ell(g)-\ell(g')\no_{L_2(\nu_Y)},
\eeqnn
and moreover:
\beqnn
\sup_{g\in\mathcal{G}}\no \ell_\lambda(g)\no_\infty\lesssim  \prod_{i=1}^d \lambda_i^{-\beta_i-1/2}.
\eeqnn
\end{lemma}
\begin{proof}
We have in dimension $d=1$ for simplicity, using the boundedness assumptions:
\beqnn
\no \ell_\lambda(g)-\ell_\lambda(g')\no_{L_2(\tilde{P})}^2
& = & \sum_{y\in\mathcal{Y}}p_y\int_\mathcal{\tilde{X}} \left[ \int_\mathcal{X}\frac{1}{\lambda} \mathcal{K}_\eta \left(\frac{z-x}{\lambda}\right) (\ell(g(x),y))-\ell(g'(x),y)))dx \right]^2 Af_y(z)dz \\
& = &\sum_{y\in\mathcal{Y}}p_y \int_{\tilde{\mathcal{X}}} \left[ \frac{1}{\lambda}\mathcal{K}_\eta(\frac{\cdot}{\lambda})*(\ell(g(\cdot),y)-\ell(g'(\cdot),y))(z) \right]^2 Af_y(z)dz \\
& \leq &C \sum_{y\in\mathcal{Y}}p_y \int_\mathcal{\tilde{X}} \frac{1}{\lambda^2}|\mathcal{F}[\mathcal{K}_\eta(\frac{\cdot}{\lambda})](t)|^2|\mathcal{F}[\ell(g(\cdot),y)-\ell(g'(\cdot),y)](t)|^2 dt \\
& \leq & C'\lambda^{-2\beta} \no \ell(g)-\ell(g')\no_{L_2(\nu_Y)}^2,
\eeqnn 
where we use in last line the following inequalities:
\beqnn
\frac{1}{\lambda^2} \left| \mathcal{F}[\mathcal{K}_\eta(./\lambda)](s) \right|^2 =  \left| \mathcal{F}[\mathcal{K}_\eta](s\lambda) \right|^2  \leq \sup_{t\in \mathbb{R}} \left| \frac{\mathcal{F}[\mathcal{K}](t\lambda)}{\mathcal{F}[\eta](t)} \right|^2 \leq \sup_{t\in [-\frac{K}{\lambda},\frac{K}{\lambda}]}  C\left| \frac{1}{\mathcal{F}[\eta](t)} \right|^2 \leq C \lambda^{-2\beta},
\eeqnn
provided that $\FF[\mathcal{K}]$ is compactly supported. 

By the same way, the second assertion holds since if $\ell(g(\cdot),y)\in L^2(\mathcal{X})$:
\begin{eqnarray*}
\sup_{(z,y)} | \ell_\lambda(g,(z,y))|
& \leq & \sup_{(z,y)} \int_{\mathcal{X}} \left| \frac{1}{\lambda}\mathcal{K}_\eta \left(\frac{z-x}{\lambda} \right)\ell(g(x),y))\right| dx \\
& \leq & C\sup_{z\in \mathcal{X}} \sqrt{\int_{\mathcal{X}} \left|\frac{1}{\lambda} \mathcal{K}_\eta \left(\frac{z-x}{\lambda} \right)\right|^2 dx} \\
& \leq & \lambda^{-\beta-1/2}.
\end{eqnarray*}
A straightforward generalization leads to the $d$-dimensional case.
\end{proof}
The last step is to get an approximation function for the class $\{\ell_\lambda(g),g\in\mathcal{G}\}$ with the following lemma:
\begin{lemma}
\label{biashold}
Suppose \textbf{(R1)} holds and $\mathcal{K}_\eta(t) = \FF^{-1}\left[ \frac{\FF[\mathcal{K}](\cdot)}{\FF[\eta](\cdot/\lambda)}\right]$ such that $K$ is a kernel of order $\gamma$ with respect to the Lebesgue measure. Then if $\{\ell(g)-\ell(g'),g,g'\in\mathcal{G}\}$ is Bernstein with parameter $\kappa\geq  1$, we have:
\beqnn
\forall g,g'\in\mathcal{G}, (R_{\ell}^\lambda-R_{\ell})(g-g')\leq a(\lambda)+r(R_{\ell}(g)-R_{\ell}(g')),
\eeqnn
where 
$$
a(\lambda)=C\sum_{i=1}^d\lambda_i^{\frac{2\kappa\gamma}{2\kappa-1}}\mbox{ and }r=\frac{1}{2\kappa}.
$$
Moreover, if $|\ell(g(x),y)-\ell(g'(x),y)|=|\ell(g(x),y)-\ell(g'(x),y)|^2$ and $\kappa>1$, we have:
$$
a(\lambda)=C\sum_{i=1}^d\lambda_i^{\frac{\kappa\gamma}{\kappa-1}}\mbox{ and }r=\frac{1}{\kappa}.
$$
\end{lemma}
\begin{proof}
We consider the case $d=1$ fro simplicity. Using the elementary property $\E K_\eta\left(\frac{Z-x}{\lambda}\right)=\E K\left(\frac{X-x}{\lambda}\right)$, gathering with Fubini, we can write:
\beqnn
(R_{\ell}^\lambda-R_{\ell})(g-g')
&=&\sum_{y\in\mathcal{Y}}p_y\int_{\mathcal{X}^2}K(u)(\ell(g(x),y)-\ell(g'(x),y))\left(f_y(x+\lambda u)-f_y(x)\right)dudx.
\eeqnn
Now since the $f_y$'s has $l=\lfloor\gamma\rfloor$ derivatives, there exists $\tau\in ]0,1[$ such that:
\beqnn
\int_\mathcal{X}K(u)\left(f_y(x+\lambda u)-f_y(x)\right)du&\leq& \int_{\mathcal{X}}K(u)\left(\sum_{k=1}^{l-1}\frac{f_y^{(k)}(x)}{k!}(\lambda u)^k+\frac{f^{(l)}(x+\tau\lambda u)}{l!}(\lambda u)^l\right)du\\
&\leq & \int_{\mathcal{X}}K(u)\left(\frac{(\lambda u)^l}{l!}(f^{(l)}(x+\tau\lambda u)-f^{(l)}(x))\right)du\\
&\leq & \int_{\mathcal{X}}\frac{L(\lambda u\tau)^{\gamma}}{l!}du\leq C\lambda^\gamma,
\eeqnn
where we use in last line the H\"older regularity of  the $f_y$'s and that $\mathcal{K}$ is a kernel of order $l=\lfloor\gamma\rfloor$. \\
Using the Bernstein assumption, one gets:
\beqnn
(R_{\ell}^\lambda-R_{\ell})(g-g')
&\leq&C\lambda^\gamma\sum_{y\in\mathcal{Y}}p_y\int_{\mathcal{X}}|\ell(g(x),y)-\ell(g'(x),y)|dx.\\
&\leq&C\lambda^\gamma\sqrt{\sum_{y\in\mathcal{Y}}p_y\left(\int_{\mathcal{X}}|\ell(g(x),y)-\ell(g'(x),y)|dx\right)^2}\\
&\leq &C \no \ell(g)-\ell(g')\no_{L_2(\nu)}\lambda^\gamma\\
&\leq & C\lambda^\gamma \left(R_{\ell}(g)-R_{\ell}(g')\right)^{\frac{1}{2\kappa}}\\
&\leq & C\lambda^{\frac{2\kappa\gamma}{2\kappa-1}}+\frac{1}{2\kappa} \left(R_{\ell}(g)-R_{\ell}(g')\right),
\eeqnn
where we use in last line Young's inequality:
$$
xy^r\leq ry+x^{1/1-r},\forall r<1,
$$
with $r=\frac{1}{2\kappa}$.\\
For the second statement, if $|\ell(g(x),y)-\ell(g'(x),y)|=|\ell(g(x),y)-\ell(g'(x),y)|^2$ and $\kappa>1$, it is straightforward that $2\kappa$ can be replaced by $\kappa$ to get the result.
\end{proof}
\begin{proof}[Proof of Theorem 2]
The proof is a straightforward application of Theorem 1. From Lemma \ref{liphold} and Lemma \ref{biashold}, condition \eqref{biascontrol} in Theorem 1 can be written:
$$
\sum_{i=1}^d\lambda_i^{\frac{2\kappa\gamma}{2\kappa-1}}\lesssim \left(\frac{\Pi_{i=1}^d\lambda_i^{-\beta_i}}{\sqrt{n}}\right)^{\frac{2\kappa}{2\kappa+\rho -1}}\Leftrightarrow\forall i=1,\ldots,d \,\,\lambda_i\lesssim n^{-\frac{2\kappa-1}{2\gamma(2\kappa+\rho-1)+2(2\kappa-1)\bar{\beta}}}.
$$
Applying Theorem 1 with a smoothing parameter $\lambda$ such that equalities hold above gives the rates of convergence. 
\end{proof}
\subsection{Proof of Theorem \ref{svdtheo}}
First step is to check that the estimation procedure described in Section 3.2 gives rise to a LB-class with respect to $\nu_Y$ with the following lemma.
\begin{lemma}
\label{lipsvd}
Suppose \textbf{(A2)} holds and $l(g(\cdot),y)\in L_2(\nu)$ for any $y\in\mathcal{Y}$. Then we have:
\beqnn
\no \ell_\lambda(g)-\ell_\lambda(g')\no_{L_2(\tilde{P})}\lesssim N^{\beta}\no \ell(g)-\ell(g')\no_{L_2(\nu_Y)},
\eeqnn
and moreover:
\beqnn
\sup_{g\in\mathcal{G}}\no \ell_\lambda(g)\no_\infty\lesssim  N^{\beta+1/2}.
\eeqnn
\end{lemma}
\begin{proof}
The proof follows the proof of Lemma \ref{liphold}. We have in dimension $d=1$ for simplicity since $(\phi_k)_{k\in\N}$ is an orthonormal basis and using the boundedness assumptions over the $f_y$'s:
\beqnn
\no \ell_N(g)-\ell_N(g')\no_{L_2(\tilde{P})}^2\hspace{-0.3cm}
& = &\hspace{-0.3cm} \sum_{y\in\mathcal{Y}}p_y\int_\mathcal{\tilde{X}} \left( \sum_{k=1}^Nb_k^{-1}\int_\mathcal{X}\phi_k(z)\phi_k(x) (\ell(g(x),y))-\ell(g'(x),y)))\nu(dx) \right)^2 Af_y(z)\nu(dz) \\
& \lesssim &\sum_{y\in\mathcal{Y}}p_y\sum_{k=1}^Nb_k^{-2}\int_\mathcal{\tilde{X}}\phi_k(z)^2\left(\int_\mathcal{X} (\ell(g(x),y))-\ell(g'(x),y)))\phi_k(x)\nu(dx) \right)^2 \nu(dz)  \\
& \leq&C N^{2\beta}\sum_{y\in\mathcal{Y}}p_y\sum_{k=1}^N\left(\int_\mathcal{X} (\ell(g(x),y))-\ell(g'(x),y)))\phi_k(x)\nu(dx) \right)^2  \\
& \leq & CN^{2\beta} \no \ell(g)-\ell(g')\no_{L_2(\nu_Y)}^2.
\eeqnn 
By the same way, the second assertion holds since if $\ell(g)\in L^2(\nu)$:
\begin{eqnarray*}
\sup_{(z,y)} | \ell_\lambda(g,(z,y))|
& \leq & \sup_{(z,y)}\left|\sum_{k=1}^N b_k^{-1}\int_{\mathcal{X}} \phi_k(x)\phi_k(z)\ell(g,(x,y)) \nu(dx)\right| \\
& \leq & \sup_{(z,y)} \sqrt{\sum_{k=1}^N b_k^{-2}} \sqrt{\sum_{k=1}^N\left(\int \phi_k(x)\ell(g,(x,y))\nu(dx)\right)^{2}\phi_k(z)^2}\\
& \leq & C  N^{\beta+1/2}.
\end{eqnarray*}
\end{proof}
The last step is to control the bias term of the procedure with the following lemma:
\begin{lemma}
\label{biassvd}
Suppose \textbf{(R2)} holds and $\{\ell(g)-\ell(g'),g,'\in\mathcal{G}\}$ is Bernstein with parameter $\kappa\geq 1$. Then we have:
\beqnn
\forall g,g'\in\mathcal{G}, (R_{\ell}^\lambda-R_{\ell})(g-g')\leq a(\lambda)+r(R_{\ell}(g)-R_{\ell}(g'))^2,
\eeqnn
where 
$$
a(N)=C\sum_{i=1}^dN_i^{-\frac{2\kappa\gamma}{2\kappa-1}}\mbox{ and }r=\frac{1}{2\kappa}.
$$
Moreover, if $|\ell(g(x),y)-\ell(g'(x),y)|=|\ell(g(x),y)-\ell(g'(x),y)|^2$ and $\kappa>1$, we have:
$$
a(N)=C\sum_{i=1}^dN_i^{-\frac{\kappa\gamma}{\kappa-1}}\mbox{ and }r=\frac{1}{\kappa}.
$$
\end{lemma}
\begin{proof}
We first write, with $E_{Z^y}\hat{\theta}_k^y=\theta_k^y\int_{\mathcal{X}}f_y(x)\phi_k(x)\nu(dx)$:
\beqnn
R_{\ell}^N(g)=\E R_n^N(g)&=&\E \int_\mathcal{X} \ell(g(x),y)\sum_{k=1}^{N}\hat{\theta}_k^y\phi_k(x)\nu(dx)\\
&=&\sum_{y\in\mathcal{Y}}p_y\int_\mathcal{X} \ell(g(x),y)\sum_{k=1}^{N}\E_{Z^y}\hat{\theta}_k^y\phi_k(x)\nu(dx)\\
&=&\sum_{y\in\mathcal{Y}}p_y\int_\mathcal{X} \ell(g(x),y)\sum_{k=1}^{N}\theta_k^y\phi_k(x)\nu(dx)
\eeqnn
Hence we have:
\beqnn
(R_{\ell}^\lambda-R_{\ell})(g-g')&=&\sum_{y\in\mathcal{Y}}p_y\int_{\mathcal{X}}(\ell(g(x),y)-\ell(g'(x),y))\left(\sum_{k=1}^{N}\theta_k^y\phi_k(x)-\sum_{k\geq 1}\theta_k^y\phi_k(x)\right)\nu(dx)\\
&=&\sum_{y\in\mathcal{Y}}p_y\int_{\mathcal{X}}(\ell(g'(x),y)-\ell(g(x),y))\sum_{k>N}\theta_k^y\phi_k(x)\nu(dx).
\eeqnn
Using Cauchy-Schwarz twice, we have since $(\phi_k)_{k\in\N}$ in an orthonormal basis and provided that $f_y\in\Theta(\gamma,L)$: 
\beqnn
|(R_{\ell}^\lambda-R_{\ell})(g-g')|&\leq &\sqrt{\sum_{y\in\mathcal{Y}}p_y\left(\int_{\mathcal{X}}(\ell(g'(x),y)-\ell(g(x),y))\phi_k(x)\nu(dx)\right)^2}\sqrt{\sum_{y\in\mathcal{Y}}p_y\left(\sum_{k>N}\theta_k^y\right)^2}\\
&\leq &\sqrt{\sum_{y\in\mathcal{Y}}p_y\int_{\mathcal{X}}(\ell(g(x),y)-\ell(g'(x),y))^2\nu(dx)\int_\mathcal{X}\phi_k^2(x)\nu(dx)}\sqrt{\sum_{y\in\mathcal{Y}}p_y\left(\sum_{k>N}\theta_k^y\right)^2}\\
&\leq &C\no \ell(g)-\ell(g')\no_{L_2(\nu_Y)}\sum_{y\in\mathcal{Y}}p_yN^{-\gamma}\sqrt{\sum_{k>N}(\theta_k^y)^2k^{2\gamma}}\\
&\leq &C\left(R_{\ell}(g)-R_{\ell}(g')\right)^{\frac{1}{2\kappa}}\sum_{y\in\mathcal{Y}}p_yN^{-\gamma}\sqrt{\sum_{k>N}(\theta_k^y)^2k^{2\gamma}}\\
&\leq &C\left(R_{\ell}(g)-R_{\ell}(g')\right)^{\frac{1}{2\kappa}}N^{-\gamma}.
\eeqnn
We conclude the proof using Young's inequality exactly as in Lemma \ref{biashold}.
\end{proof}
\begin{proof}[Proof of Theorem 3]
The proof is a straightforward application of Theorem 1. From Lemma \ref{lipsvd} and Lemma \ref{biassvd}, condition \eqref{biascontrol} in Theorem 1 can be written:
$$
N^{\frac{-2\kappa\gamma}{2\kappa-1}}\lesssim \left(\frac{N^{\beta}}{\sqrt{n}}\right)^{\frac{2\kappa}{2\kappa+\rho -1}}\Leftrightarrow N\lesssim n^{\frac{2\kappa-1}{2\gamma(2\kappa+\rho-1)+2(2\kappa-1)\beta}}.
$$
Applying Theorem 1 with a smoothing parameter $N$ such thatan  equality holds above gives the rates of convergence. 
\end{proof}
\subsection{Proof of Lemma \ref{dudleynoisy}}
The proof uses the maximal inequality presented in \cite{wvdv} to the class:
\beqnn
\mathcal{F}=\{\ell_\lambda(g)-\ell_\lambda(g'),g,g'\in\GG:P(\ell(g)-\ell(g'))^2\leq\delta^2\}.
\eeqnn
Indeed from Theorem 2.14.2 of \cite{wvdv}, we can write, $\forall \eta>0$:
\beqn
\label{maxineq}
\tilde{\omega}_n(\GG,\delta,\mu)&=&\E\sup_{g,g'\in\GG:\Arrowvert \ell(g)-\ell(g')\Arrowvert_{L_2(\mu)}^2\leq \delta^2}\left|(\tilde{P}_n-\tilde{P})(\ell_\lambda(g)-\ell_\lambda(g'))\right| \nonumber\\
&\leq &\frac{\no F\no^2_{L_2(\tilde{P})}}{\sqrt{n}}\int_0^\eta\sqrt{1+\mathcal{H}_{B}(\mathcal{F},\epsilon\no F\no^2_{L_2(\tilde{P})},L_2(\mu))}d\epsilon\nonumber\\
&+&\frac{\sup_{f\in\FF}\no f\no_{L_2(\tilde{P})}}{\sqrt{n}}\sqrt{1+\mathcal{H}_B(\FF,\eta\no F\no^2_{L_2(\tilde{P})},L_2(\mu))}
\eeqn
where $F(z,y)=\sup_{f\in\FF}|\ell_\lambda(g,z,y)-\ell_\lambda(g',z,y)|$ is the enveloppe function of the class $\FF$. Since $\{\ell_\lambda(g),g\in\GG\}$ is a LB-class with bounded constant $K(\lambda)$:
\beqnn
\no F\no^2_{L_2(\tilde{P})}&=&\int F^2(z)P(dz,dy)\\
&=&\sum_{y\in\mathcal{Y}}p_y\int\left(\sup_{f\in\FF}|\ell_\lambda(g,z,y)-\ell_\lambda(g',z,y)|\right)^2 Af_y(z)\nu(dz)\\
&\lesssim & K(\lambda)^2.\\
\eeqnn
Moreover, we have since $\{\ell_\lambda(g),g\in\GG\}$ is a LB-class with respect to $\mu$ with Lipschitz constant $c(\lambda)$:
\beqnn
\mathcal{H}_{B}(\{\ell(g),\,g\in\mathcal{G}\},\epsilon,L_2(\mu))\leq c\epsilon^{-2\rho}\Rightarrow \mathcal{H}_{B}(\FF,\epsilon,L_2(\tilde{P}))\lesssim c(\lambda)^{2\rho}\epsilon^{-2\rho}.
\eeqnn
Hence, we have in \eqref{maxineq}, choosing $\eta=\frac{c(\lambda)}{K(\lambda)^2}\delta$:
\beqnn
\tilde{\omega}_n(\GG,\delta)
&\lesssim &\frac{K(\lambda)^2}{\sqrt{n}}\int_0^\eta\sqrt{1+\epsilon^{-2\rho}K(\lambda)^{-4\rho}c(\lambda)^{2\rho}}d\epsilon+\frac{c(\lambda)\delta}{\sqrt{n}}\sqrt{1+\eta^{-2\rho}K(\lambda)^{-4\rho}c(\lambda)^{2\rho}}\\
&\lesssim &\frac{\eta K(\lambda)^2}{\sqrt{n}}+\frac{\eta^{1-\rho}K(\lambda)^{2(1-\rho)}c(\lambda)^\rho}{\sqrt{n}}+\frac{c(\lambda)\delta}{\sqrt{n}}+\frac{c(\lambda)^{1+\rho}\eta^{-\rho}K(\lambda)^{-2\rho}\delta}{\sqrt{n}}\\
&\lesssim& \frac{\eta^{1-\rho}K(\lambda)^{2(1-\rho)}c(\lambda)^\rho}{\sqrt{n}}+\frac{c(\lambda)^{1+\rho}\eta^{-\rho}K(\lambda)^{-2\rho}\delta}{\sqrt{n}} \frac{c(\lambda)}{\sqrt{n}}\delta^{1-\rho},
\eeqnn
provided that $\delta\leq 1$.

\bibliographystyle{plainnat}
\bibliography{referencenoisy}

\begin{thebibliography}{29}
\providecommand{\natexlab}[1]{#1}
\providecommand{\url}[1]{\texttt{#1}}
\expandafter\ifx\csname urlstyle\endcsname\relax
  \providecommand{\doi}[1]{doi: #1}\else
  \providecommand{\doi}{doi: \begingroup \urlstyle{rm}\Url}\fi

\bibitem[Audibert and Tsybakov(2007)]{AT}
J-Y. Audibert and A.B. Tsybakov.
\newblock Fast learning rates for plug-in classifiers.
\newblock \emph{The Annals of statistics}, 35:\penalty0 608--633, 2007.

\bibitem[Bartlett and Mendelson(2006)]{empimini}
P.L. Bartlett and S.~Mendelson.
\newblock Empirical minimization.
\newblock \emph{Probability Theory and Related Fields}, 135 (3):\penalty0
  311--334, 2006.

\bibitem[Blanchard et~al.(2008)Blanchard, Bousquet, and Massart]{svm}
G.~Blanchard, O.~Bousquet, and P.~Massart.
\newblock Statistical performance of support vector machines.
\newblock \emph{The Annals of Statistics}, 36 (2):\penalty0 489--531, 2008.

\bibitem[Bousquet(2002)]{bousquet}
O.~Bousquet.
\newblock A bennet concentration inequality and its application to suprema of
  empirical processes.
\newblock \emph{C.R. Acad. SCI. Paris Ser. I Math}, 334:\penalty0 495--500,
  2002.

\bibitem[Butucea(2007)]{butucea}
C.~Butucea.
\newblock goodness-of-fit testing and quadratic functionnal estimation from
  indirect observations.
\newblock \emph{The Annals of Statistics}, 35:\penalty0 1907--1930, 2007.

\bibitem[Cavalier(2008)]{cavaliersurvey}
L.~Cavalier.
\newblock Nonparametric statistical inverse problems.
\newblock \emph{Inverse Problems}, 24:\penalty0 1--19, 2008.

\bibitem[Devroye et~al.(1996)Devroye, {Gy\"orfi}, and Lugosi]{jaune}
L.~Devroye, L.~{Gy\"orfi}, and G.~Lugosi.
\newblock \emph{A Probabilistic Theory of Pattern Recognition.}
\newblock Springer-Verlag, 1996.

\bibitem[{\textsc{E}}ngl et~al.(1996){\textsc{E}}ngl, {\textsc{H}}ank, and
  {\textsc{N}}eubauer]{engle}
H.W. {\textsc{E}}ngl, M.~{\textsc{H}}ank, and A.~{\textsc{N}}eubauer.
\newblock \emph{Regularization of Inverse Problems}.
\newblock Kluwer Academic Publishers Group, Dordrecht, 1996.

\bibitem[Fan(1991)]{Fan}
J.~Fan.
\newblock On the optimal rates of convergence for nonparametric deconvolution
  problems.
\newblock \emph{The Annals of Statistics}, 19:\penalty0 1257--1272, 1991.

\bibitem[Klemela and Mammen(2010)]{klemela}
J~Klemela and E.~Mammen.
\newblock Empirical risk minimization in inverse problems.
\newblock \emph{The Annals of Statistics}, 38 (1):\penalty0 482--511, 2010.

\bibitem[Koltchinskii(2006)]{kolt}
V.~Koltchinskii.
\newblock Local rademacher complexities and oracle inequalties in risk
  minimization.
\newblock \emph{The Annals of Statistics}, 34 (6):\penalty0 2593--2656, 2006.

\bibitem[Lecu\'e and Mendelson(2012)]{nonexactlecue}
G.~Lecu\'e and S.~Mendelson.
\newblock General non-exact oracle inequalities for classes with a
  subexponential envelope.
\newblock Futur paper, 2012.

\bibitem[Lederer and van~de Geer(2012)]{lederer}
Y.~Lederer and S.~van~de Geer.
\newblock New concentration inequalities for suprema of empirical processes.
\newblock Submitted, 2012.

\bibitem[Loustau(2009)]{loustau2}
S.~Loustau.
\newblock Penalized erm over besov spaces.
\newblock \emph{Electronic journal of Statistics}, 3:\penalty0 824--850, 2009.

\bibitem[Loustau and Marteau(2011)]{pinkfloyds}
S.~Loustau and C.~Marteau.
\newblock Discriminant analysis with errors in variables.
\newblock \emph{http://hal.archives-ouvertes.fr/hal-00660383}, 2011.

\bibitem[Mammen and Tsybakov(1999)]{mammen}
E.~Mammen and A.B. Tsybakov.
\newblock Smooth discrimination analysis.
\newblock \emph{The Annals of Statistics}, 27 (6):\penalty0 1808--1829, 1999.

\bibitem[Massart(2000)]{toulouse}
P.~Massart.
\newblock Some applications of concentration inequalities to statistics.
\newblock \emph{Ann. Fac. Sci. Toulouse Math.}, 9 (2):\penalty0 245--303, 2000.

\bibitem[Massart and N\'ed\'elec(2006)]{nedelec}
P.~Massart and E.~N\'ed\'elec.
\newblock Risk bounds for statistical learning.
\newblock \emph{The Annals of Statistics}, 34 (5):\penalty0 2326--2366, 2006.

\bibitem[Meister(2009)]{meister}
A.~Meister.
\newblock \emph{Deconvolution problems in nonparametric statistics}.
\newblock Springer-Verlag, 2009.

\bibitem[Mendelson(2003)]{mendelsonkernel}
S.~Mendelson.
\newblock On the performance of kernel classes.
\newblock \emph{Journal of Machine Learning Research}, 4:\penalty0 759--771,
  2003.

\bibitem[Talagrand(1996)]{talagrand}
M.~Talagrand.
\newblock New concentration inequalities in product spaces.
\newblock \emph{Invent. Math}, 126:\penalty0 505--563, 1996.

\bibitem[Tsybakov(2004{\natexlab{a}})]{booktsybakov}
A.B. Tsybakov.
\newblock \emph{Introduction \`a l'estimation non-param\'etrique}.
\newblock Springer-Verlag, 2004{\natexlab{a}}.

\bibitem[Tsybakov(2004{\natexlab{b}})]{tsybakov2004}
A.B. Tsybakov.
\newblock Optimal aggregation of classifiers in statistical learning.
\newblock \emph{The Annals of Statistics}, 32 (1):\penalty0 135--166,
  2004{\natexlab{b}}.

\bibitem[Tsybakov and van~de Geer(2005)]{tsybakov2005}
A.B. Tsybakov and S.A. van~de Geer.
\newblock Square root penalty: adaptation to the margin in classification and
  in edge estimation.
\newblock \emph{The Annals of Statistics}, 33 (3):\penalty0 1203--1224, 2005.

\bibitem[van~de Geer(2000)]{vdg}
S.~van~de Geer.
\newblock \emph{Empirical Processes in M-estimation}.
\newblock Cambridge University Press, 2000.

\bibitem[van~der Vaart and Wellner(1996)]{wvdv}
A.~W. van~der Vaart and J.~A. Wellner.
\newblock \emph{Weak convergence and Empirical Processes. With Applications to
  Statistics}.
\newblock Springer Verlag, 1996.

\bibitem[Vapnik(1982)]{vapnik82}
V.~Vapnik.
\newblock \emph{Estimation of Dependances Based on Empirical Data}.
\newblock Springer Verlag, 1982.

\bibitem[Vapnik(2000)]{vapnik2000}
V.~Vapnik.
\newblock \emph{The Nature of Statistical Learning Theory}.
\newblock Statistics for Engineering and Information Science, Springer, 2000.

\bibitem[Williamson et~al.(2001)Williamson, Smola, and {Sch\"olkopf}]{will}
R.C. Williamson, A.J. Smola, and B.~{Sch\"olkopf}.
\newblock Generalization performance of regularization networks and support
  vector machines via entropy numbers of compact operators.
\newblock \emph{IEEE Transactions on Information Theory}, 47 (6):\penalty0
  2516--2532, 2001.

\end{thebibliography}

\end{document}